\documentclass[11pt,reqno]{amsart}
\usepackage[margin=2.5cm]{geometry}
\allowdisplaybreaks
\usepackage{mlmodern}

\usepackage[english]{babel}
\usepackage[latin1]{inputenc}
\usepackage[T1]{fontenc}
\usepackage{amssymb,amsmath,amstext,amsfonts,amsthm,mathrsfs}
\usepackage{mathtools}
\usepackage{relsize}
\usepackage[colorlinks=true, urlcolor=blue, linkcolor=blue, citecolor=blue]{hyperref}
\usepackage{todonotes}
\usepackage[normalem]{ulem}
\usepackage{graphicx}
\usepackage{arydshln}
 
\definecolor{OliveGreen}{rgb}{0,0.6,0}

\numberwithin{equation}{section}
\theoremstyle{plain}
\newtheorem*{theorem*}{Theorem}
\newtheorem{theorem}{Theorem}
\numberwithin{theorem}{section}
\newtheorem{proposition}[theorem]{Proposition}
\newtheorem{lemma}[theorem]{Lemma}
\newtheorem{corollary}[theorem]{Corollary}
\newtheorem{conjecture}[theorem]{Conjecture}
\newtheorem{problem}[theorem]{Problem}

\newtheorem{remark}[theorem]{Remark}

\theoremstyle{definition}
\newtheorem{definition}[theorem]{Definition}
\newtheorem{example}[theorem]{Example}

\newcommand{\C}{\mathbb{C}}

\newcommand{\K}{\mathbb{K}}
\newcommand{\bS}{\mathbb{S}}
\newcommand{\bG}{\mathbb{G}}
\newcommand{\Q}{\mathbb{Q}}
\newcommand{\Z}{\mathbb{Z}}
\newcommand{\PP}{\mathbb{P}}
\newcommand{\V}{\mathbb{V}}

\newcommand{\R}{\mathbb{R}}
\newcommand{\sA}{\mathcal{A}}
\newcommand{\sS}{\mathcal{S}}
\newcommand{\sB}{\mathcal{B}}
\newcommand{\sC}{\mathcal{C}}
\newcommand{\sD}{\mathcal{D}}
\newcommand{\sE}{\mathcal{E}}

\newcommand{\sQ}{\mathcal{Q}}
\newcommand{\sL}{\mathcal{L}}
\newcommand{\sT}{\mathcal{T}}
\newcommand{\sM}{\mathcal{M}}
\newcommand{\sN}{\mathcal{N}}
\newcommand{\sO}{\mathcal{O}}
\newcommand{\sU}{\mathcal{U}}
\newcommand{\sV}{\mathcal{V}}
\newcommand{\sW}{\mathcal{W}}
\newcommand{\sX}{\mathcal{X}}
\newcommand{\sY}{\mathcal{Y}}
\newcommand{\sZ}{\mathcal{Z}}
\newcommand{\bn}{\mathbf{n}}
\newcommand{\bd}{\mathbf{d}}
\newcommand{\mR}{\mathsmaller{\mathbb{R}}}
\newcommand{\mT}{\mathsmaller{\mathsf{T}}}
\newcommand{\mC}{\mathsmaller{\C}}
\newcommand{\balpha}{\boldsymbol{\alpha}}
\newcommand{\bomega}{\boldsymbol{\omega}}
\newcommand{\codim}{\mathrm{codim}}
\newcommand{\dd}{\textrm{\normalfont{d}}}
\DeclareMathOperator{\dist}{dist}
\DeclareMathOperator{\pr}{pr}
\DeclareMathOperator{\EDeg}{EDdegree}
\DeclareMathOperator{\gEDeg}{gEDdegree}
\DeclareMathOperator{\EDef}{EDdefect}

\makeatletter
\@namedef{subjclassname@2020}{%
	\textup{2020} Mathematics Subject Classification}
\makeatother

\author[K. Kozhasov]{Khazhgali Kozhasov}
\address{Friedrich-Schiller-Universit\"at Jena, Institut f\"ur Mathematik, Ernst-Abbe-Platz 2 07737 Jena, Germany}
\curraddr{Laboratoire J.-A. Dieudonn\'e \\ Universit\'e C\^ote d'Azur \\ Parc Valrose, 06108 Nice, France }
\email{khazhgali.kozhasov@univ-cotedazur.fr}

\author[A. Muniz]{Alan Muniz}
\address{Universidade Estadual de Campinas (UNICAMP) \\ Instituto de Matem\'atica, Estat\'istica e Computa\c{c}\~ao Cient\'ifica (IMECC) \\ Rua S\'ergio Buarque de Holanda, 651 \\ 13083-970 Campinas-SP, Brazil}
\curraddr{Departamento de Matem\'atica \\ Centro de Ci\^encias Exatas e da Natureza \\ Universidade Federal de Pernambuco \\ Recife - PE, CEP 50740-560, Brazil }
\email{alan.nmuniz@ufpe.br}

\author[Y. Qi]{Yang Qi}
\address{INRIA and CMAP, \'Ecole Polytechnique, IP Paris, CNRS, France}
\email{yang.qi@inria.fr}

\author[L. Sodomaco]{Luca Sodomaco}
\address{Department of Mathematics, KTH Royal Institute of Technology, SE-100 44 Stockholm, Sweden}
\curraddr{Max Planck Institute for Mathematics in the Sciences, Leipzig, Germany}
\email{luca.sodomaco@mis.mpg.de}

\subjclass[2020]{}
\keywords{}

\title[Minimal algebraic complexity of the rank-one approximation problem]{On the minimal algebraic complexity of the rank-one approximation problem for general inner products}

\date{}

\begin{document}

\begin{abstract}
We study the algebraic complexity of Euclidean distance minimization from a generic tensor to a variety of rank-one tensors. The Euclidean Distance (ED) degree of the Segre-Veronese variety counts the number of complex critical points of this optimization problem.
We regard this invariant as a function of inner products. We prove that Frobenius inner product is a local minimum of the ED degree, and conjecture that it is a global minimum. We prove our conjecture in the case of matrices and symmetric binary and $3\times 3\times 3$ tensors. We discuss the above optimization problem for other algebraic varieties, classifying all possible values of the ED degree. Our approach combines tools from Singularity Theory, Morse Theory, and Algebraic Geometry.
\end{abstract}

\maketitle

\section{Introduction}
The best rank-one approximation problem is classically studied in Multilinear Algebra and Optimization \cite{EckartYoung, Banach, Friedland2011OnBR, GKT2013, Stu22icm}. 
Given finite-dimensional real vector spaces $V^{\mR}_1, \ldots, V^{\mR}_k$ and (the square of) some norm $q$ on their tensor product $V_1^{\mR}\otimes\cdots\otimes V^{\mR}_k$, the problem consists in finding global minima of the distance function $(x_1,\ldots, x_k)\mapsto q(u-x_1\otimes\cdots\otimes x_k)$ from a given \emph{data tensor} $u\in V_1^{\mR}\otimes\cdots\otimes V^{\mR}_k$ to the set of rank-one (also known as, decomposable) tensors,
\begin{align}\label{eq:BROAP}
 \min_{x_i\in V^{\mR}_i,\, i\in[k]} q(u-x_1\otimes\cdots\otimes x_k).
\end{align}
Any minimizer in \eqref{eq:BROAP} is referred to as \emph{a best rank-one approximation of $u$} (with respect to $q$).

The above problem has numerous applications in Mathematics and other areas such as Data Science \cite{Udell2017WhyAB, GJZh, kohn2022geometry,kohn2023geometry} and Signal Processing \cite{KR, DELATHAUWER200431, Co, QCL16, SDFHPF}. 
Furthermore, optimization of a homogeneous polynomial over the Euclidean sphere in $V^{\mR}$ can be modeled as \eqref{eq:BROAP} for a \emph{symmetric} tensor $u\in V^{\mR}\otimes\cdots\otimes V^{\mR}$, see \cite{Nie2013SemidefiniteRF, Kozhasov2017OnFR, KlerkLaurent}.
Furthermore, in Quantum Information Theory, an analog of \eqref{eq:BROAP} for complex tensors represents one of the possible measures of \emph{quantum entanglement of multipartite states} \cite{Wei2003GeometricMO, Derksen2017TheoreticalAC, DM2018}. 

In the case of matrices (i.e., tensors of order $k=2$), problem \eqref{eq:BROAP} is, in general, much more manageable. Thus, when $q$ is associated with the \emph{Frobenius} (also known as \emph{trace}) \emph{inner product}, the celebrated Eckart-Young theorem \cite{EckartYoung} gives a recipe for finding a best rank-one approximation to a given \emph{data matrix} $u\in V^{\mR}_1\otimes V^{\mR}_2$: one simply needs to truncate the Singular Value Decomposition of $u$.
There is no direct way to solve the analogous problem for tensors of order $k\geq 3$, which is known to be NP-hard \cite{Nesterov, HiLi}.
Despite this and the fact that a \emph{best rank-$r$ approximation} of a high-order tensor might not exist for $r>1$ \cite{deSilva2008tensor}, problem \eqref{eq:BROAP} lies in the core of all existing low-rank approximation schemes for tensors \cite{KoldaBader, GKT2013}.
So, for example, \emph{the deflation method}, i.e., \emph{successive rank-one approximations}, is used to find a reasonable low-rank approximation of a tensor \cite{daSilva2015iterative}.
Thus, theoretical advances in understanding problem \eqref{eq:BROAP} are of great importance.

When $q$ in \eqref{eq:BROAP} is the (square of the) norm associated with an inner product $Q$ on the space of tensors $V^{\mR}_1\otimes \dots\otimes V^{\mR}_k$,
a best rank-one approximation, being a minimizer, is a critical point of the distance function to the manifold of rank-one tensors.
{\em Critical rank-one approximations} (to a data tensor $u$) are by definition critical points of the distance function in \eqref{eq:BROAP}, see \cite{FO, DH2016, DOT, HTT2023}.
These are precisely real solutions to a certain system of polynomial equations in the entries of the rank-one tensor. When the data tensor $u$ is general enough, the number of complex solutions of the same system of polynomial equations is finite and known as the \emph{Euclidean Distance (ED) degree} of the algebraic variety of complex rank-one tensors, see \cite{DHOST} and Section \ref{sec: prelim} for more details.
One also refers to it as \emph{the algebraic complexity} of the rank-one approximation problem, since the ED degree accounts for all (possibly complex) solutions to the system of polynomial equations defining a best rank-one approximation as a critical point of the distance function. In particular, it provides an upper bound on the number of critical rank-one approximations to a general data tensor.
Thus the study of ED degree is strongly motivated by and has deep applications in both Optimization and Data Science, namely, a target function with fewer critical points is potentially easier to minimize via local optimization methods such as gradient descent, Newton method, and their variations \cite{ge2022optimization,kohn2023function}.
In the present work, we regard the ED degree as a function of the inner product $Q$ and attempt to minimize it. As we discuss below, changing the inner product may be beneficial in designing a more convenient cost function, i.e., one can hope to find $Q$ that gives rise to the smallest algebraic complexity of the rank-one approximation problem.
For this reason, we introduce the {\em ED degree map} in \eqref{eq: ED degree map}.

We conjecture that, in the minimization problem~\eqref{eq:BROAP}, the Frobenius inner product achieves the smallest possible ED degree, see Section \ref{sec: product metrics tensor spaces} and Conjecture \ref{conj: main}.
This problem turns out to be nontrivial already in the case of matrices, for which we are able to settle it.
In fact, in Section \ref{sec: minimum ED degree Segre-Veronese} we show that, with respect to an arbitrary fixed inner product, any sufficiently general data matrix of size $n\times m$ possesses at least $\min(n,m)$ many critical rank-one approximations. Hence, the algebraic complexity of the best rank-one approximation problem for matrices cannot be smaller than $\min(n,m)$, which equals the ED degree of the Frobenius inner product.
Our result can be considered a qualitative version of the Eckart-Young theorem with respect to a general inner product. In practical terms, it asserts that, topologically, the distance function from a general data matrix to the manifold of rank-one matrices cannot be too simple, as it has at least one critical point of each \emph{Morse index}, see Theorem \ref{thm: lower bound Q-distance function}.

Symmetric tensors of order $d$ are in one-to-one correspondence with degree $d$ homogeneous polynomials, and (up to a sign) a symmetric rank-one tensor corresponds to the $d$-th power of a linear form.
It has been noticed in~\cite[Example~2.7]{DHOST} that the ED degree of the variety consisting of symmetric rank-one tensors would change under different inner products.
As a special case, the Frobenius inner product of two symmetric tensors equals the so-called \emph{Bombieri-Weyl} inner product of the associated homogeneous polynomials.
In this case, a best rank-one approximation of a symmetric data tensor $u$ can always be chosen to be symmetric \cite{Banach, Friedland2011BestRO} and the minimization problem \eqref{eq:BROAP} is equivalent to optimizing the associated homogeneous polynomial over the unit sphere in $V^\R$, see \cite[Thm. $2.19$]{QiLuo}.
The algebraic complexity, given by the ED degree of the variety of $d$-th powers of complex linear forms, provides a sharp upper bound on the number of critical points of the mentioned polynomial optimization problem, see \cite{Kozhasov2017OnFR}.
A particular case of Conjecture~\ref{conj: main} then asserts that the Bombieri-Weyl inner product achieves the smallest possible ED degree among all inner products on the space of homogeneous polynomials, see Theorem \ref{thm: symmetric_matrices}.

In addition to proposing various optimization problems with different complexities, changing inner products would give rise to various normal distributions on the space of tensors, thus causing different behaviors of a random tensor. More precisely, the space of real order-$d$ symmetric tensors, denoted by $S^d V^{\R}$, admits a $\lceil d\,/\,2\rceil$-dimensional family of orthogonally invariant inner products generalizing $Q_F$ \cite{Kostlan}, and each such inner product admits a unique standard normal distribution on $S^d V^\R$ associated with it. Kostlan \cite{Kostlan} computed the expected volume of a random projective hypersurface given by a polynomial sampled under these normal distributions on $S^d V^\R$. It follows from \cite[Thm $5.4$]{Kostlan} that there exist invariant scalar products giving rise to a smaller (respectively, larger) expected volume of a random hypersurface than $Q_F$. This shows that the Frobenius inner product, although the most natural inner product to take, does not always provide us the most expected result. In this regard, it would be interesting to compare the so-called \emph{average ED degree} (see \cite[Section $4$]{DH2016} and \cite{Breiding2017HowME}) of the variety of rank-one symmetric tensors with respect to different orthogonally invariant inner products on $S^d V^\R$.

Distance minimization is not peculiar only to varieties of rank-one matrices and tensors. More generally, one can consider an arbitrary algebraic subset $\sX^\mR$ of a real vector space $V^\mR$. Again, denoting by $q$ the (square of the) norm associated with an inner product $Q$ on $V^\mR$, we are interested in the optimization problem
\begin{equation}\label{eq: min distance function}
\min_{x\in \sX^\mR}q(u-x)\,,
\end{equation}
where $u\in V^\mR$ is a given data point. 
Again, one defines the ED degree of the algebraic variety $\sX$, the complexification of $\sX^{\mR}$, as the number of complex critical points of the distance function in \eqref{eq: min distance function} for a general $u\in V^{\mR}$. 
As was shown in \cite{DHOST}, the ED degree does not depend on $Q$ as long as it is general enough.
This number, called the \emph{generic ED degree} of $\sX$, gives an upper bound on the ED degree of $\sX$ with respect to any other inner product $Q$ on $V^{\mR}$.
The \emph{ED defect} of $\sX$ with respect to $Q$ is then defined as the difference between the generic ED degree of $\sX$ and the ED degree of $\sX$ relative to $Q$.
Conjecture \ref{conj: main} can be equivalently stated as follows: the largest ED defect of the variety of rank-one tensors is achieved for the Frobenius inner product.
An elegant formula for the ED defect was discovered in \cite{maxim2020defect}.
In Section \ref{sec: ED defects} we exploit it to analyze ED defects of some special varieties, such as quadratic hypersurfaces, rational normal curves, and Veronese surfaces. 
In some cases, we provide a complete classification of possible ED defects.
These results, in particular, give further evidence to our Conjecture \ref{conj: main} beyond the matrix case.

Another goal of this work is to link our results on the defectivity of the ED degree to the study of the {\em extended ED polynomial}, a version of the ED polynomial studied in \cite{ottaviani2020distance} where the coefficients of the inner product $Q$ are considered as additional parameters, see Definition \ref{def: extended ED polynomial}. For a fixed $Q$, the ED polynomial of $\sX$ depends only on a distance parameter $\varepsilon$ and describes the locus of data points having distance $\varepsilon$ from $\sX^\mR$. We also apply the properties of the extended ED polynomial to show in Theorem \ref{thm: ED degree map is locally constant} that the ED degree map $\Phi_\sX$ of \eqref{eq: ED degree map} is {\em lower semicontinuous}. This allows us to prove in Theorem \ref{thm: local minimality Frobenius ED degree} a ``local'' version of Conjecture \eqref{conj: main} for any Segre-Veronese variety.

\subsection*{Main contributions}
We investigate how ED degrees change with the inner product. After giving an explicit formula for the generic ED degree of Segre-Veronese varieties and presenting concrete numerical experiments, we propose Conjecture~\ref{conj: main} which states that, among all inner products on the space of tensors of a given format, the Frobenius inner product gives rise to the smallest ED degree (algebraic complexity) for the optimization  problem~\eqref{eq:BROAP}. In Section~\ref{sec: product metrics tensor spaces}, we prove Theorem~\ref{thm: local minimality Frobenius ED degree} demonstrating the ED degree induced by the Frobenius inner product gives at least the local minimum of the ED degrees. Section~\ref{sec: minimum ED degree Segre-Veronese} is devoted to the proofs of Conjecture~\ref{conj: main} for the matrix and the symmetric matrix cases. We study the ED defects in more detail for quadric hypersurfaces, rational normal curves, and Veronese surfaces in Section~\ref{sec: ED defects}. These results aim to generalize the study of the ED defects of the Segre-Veronese varieties to more general projective varieties. In Section~\ref{sec: ED polynomial}, we present a geometric description of the variety consisting of those inner products whose associated ED degrees are strictly less than the generic ED degree. This corresponds to the vanishing locus of the leading coefficient of the extended ED polynomial given in Definition \ref{def: extended ED polynomial}. As explained in Corollary \ref{cor: main conj revisited},
further study of the vanishing loci of all coefficients of the extended ED polynomial of a Segre-Veronese variety (or any other algebraic variety) would help us understand better the locus of inner products that induce smaller ED degrees.

The Macaulay2 code \cite{GS} used in this work can be found at 
\begin{center}
    \url{https://github.com/sodomal1/minimal-ED-degree.git}
\end{center}

\section*{Acknowledgements}

This work is partially supported by the Thematic Research Program {\em``Tensors: geometry, complexity and quantum entanglement''}, University of Warsaw, Excellence Initiative - Research University, and the Simons Foundation Award No. 663281 granted to the Institute of Mathematics of the Polish Academy of Sciences for the years 2021-2023. We are thankful for the support and great working conditions.
AM was supported by INCTmat/MCT/Brazil, CNPq grant number 160934/2022-2.
LS was supported by a KTH grant from the Verg Foundation and Brummer \& Partners MathDataLab.
We are also grateful to Jose Israel Rodriguez for a useful discussion.

\section{Preliminaries}\label{sec: prelim}

In this section, we briefly recall basic definitions and necessary results in the research concerning the Euclidean Distance degree.
In particular, in Subsection~\ref{sec: Morse theory} we include basic materials on Morse theory, which play important roles in the later sections. For further details from the perspective of algebraic geometry, we refer to \cite{DHOST}.

\subsection{Algebraic complexity of Euclidean Distance optimization}\label{subsec: ED degree ED polynomial}

Let $V$ be an $(N+1)$-dimensional complex vector space with a real structure. Its real part $V^\mR$ is equipped with a positive definite symmetric bilinear form $Q\colon V^\mR \times V^\mR\to\R$, hence the pair $(V^\mR, Q)$ is a Euclidean space. The quadratic form associated with $Q$ is denoted by $q$ with $q(x)=Q(x,x)$ for all $x\in V^\mR$. 
We extend $Q$ and $q$ to a symmetric bilinear form and, respectively, a quadratic form on the complex space $V$, denoting these by the same letters. 
Furthermore, we denote by $\sQ$ the complex quadric hypersurface defined by the equation $q(x)=0$, and it is commonly referred to as the {\em isotropic quadric}. We may regard $\sQ$ as an affine cone in $V$ or as a projective variety in $\PP(V)=\PP^N$.
Note that, since $Q$ is positive definite, the real locus $\sQ^\mR=\sQ\cap\PP(V^\mR)$ of $\sQ$ is empty.

For a real projective variety $\sX\subset\PP^N$ we denote by $\sX^\mR\subset\PP^N_{\mR}=\PP(V^\mR)$ its real part. The \emph{affine cone} over $\sX$ is denoted by $C\sX\subset V$. Note that $C\sX$ is a real affine variety, and the set of its real points satisfies $(C\sX)^\mR=C\sX^\mR\subset V^\mR$. Given a \emph{data point} $u\in V^\mR$, we consider the {\em distance function}
\begin{equation}\label{eq: distance function cone}
 \dist_{C\sX^\mR,u}^Q\colon C\sX^\mR\to\R\,,\quad x\mapsto q(u-x)\,.
\end{equation}
The main optimization problem is
\begin{equation}\label{eq: main optimization problem}
 \min_{x\in C\sX^\mR}\dist_{C\sX^\mR,u}^Q(x)\,.
\end{equation}
The algebraic relaxation of the above problem consists of studying all critical points of $\dist_{C\sX^\mR,u}^Q$, seen as a polynomial objective function defined over the complex variety $C\sX$ and taking complex values. Indeed, the smooth local minimizers of the distance function sit among these critical points. A point $x$ on the (Zariski open) subset $(C\sX)_{\mathrm{sm}}\subset C\sX$ of smooth points of $C\sX$ is {\em critical} for $\dist_{C\sX^\mR,u}^Q$ if the vector $u-x$ belongs to the {\em normal space} of $C\sX$ at $x$:
\begin{equation}\label{eq: def normal space}
 N_xC\sX\coloneqq\{v^*\in V \mid\text{$Q(v^*,v)=0$ for all $v\in T_xC\sX$}\}\,,
\end{equation}
where $T_xC\sX$ is the tangent space of $C\sX$ at $x$.
\begin{remark}
Instead of $C\sX$ we could consider any affine variety and then state problem \eqref{eq: main optimization problem} and define the necessary notions in the same way. However, as we deal only with projective varieties, we restrict to this setting and develop the exposition accordingly.
\end{remark}

The {\em ED correspondence} of $\sX$ is the incidence variety
\begin{equation}\label{eq: ED correspondence}
 \sE(\sX,Q)\coloneqq\overline{\left\{(x,u)\in V\times V\mid\text{$x\in(C\sX)_\mathrm{sm}$ and $u-x\in N_xC\sX$}\right\}}\,,
\end{equation}
where the closure is taken with respect to the Zariski topology on $V\times V$. Denote with $\mathrm{pr}_1$ and $\mathrm{pr}_2$ the projections of $\sE(\sX,Q)$ onto the two factors of $V\times V$.
On the one hand, the first projection $\mathrm{pr}_1\colon\sE(\sX,Q)\to C\sX$ is surjective and provides the structure of a vector bundle over $(C\sX)_{\mathrm{sm}}$, whose rank equals the codimension of $\sX\subset\PP^N$. Consequently, $\sE(\sX,Q)$ is an irreducible variety in $V\times V$ of dimension $N+1=\dim(V)$. On the other hand, the second projection $\mathrm{pr}_2 \colon\sE(\sX,Q)\to V$ is surjective, and its fibers are generically of constant cardinality. By definition, this cardinality $\EDeg_Q(\sX)$ is the {\em Euclidean Distance (ED) degree of $\sX$} with respect to $Q$, and those $x\in\sX_{\mathrm{sm}}$ for which $(x,u)\in \sE(\sX,Q)$ are referred to as {\em ED degree critical points} (relative to $u$). 
The {\em ED discriminant} $\Sigma(\sX,Q)$ is then defined as the Zariski closure in $V$ of the set of critical values of $\mathrm{pr}_2$:
\begin{align}\label{eq: ED disc}
 \Sigma(\sX,Q)\coloneqq\overline{\left\{ u\in V \mid\text{$\dim(\dd_{(x,u)}\mathrm{pr}_2(T_{(x,u)} \sE(\sX,Q)))< N+1$ for some $(x,u)\in \sE(\sX,Q)_{\mathrm{sm}}$}\right\}}\,,
\end{align}
where $\dd_{(x,u)} \mathrm{pr}_2\colon T_{(x,u)} \sE(\sX,Q)\to T_u V=V$ is the differential of $\mathrm{pr}_2$.
In particular, $u\in \Sigma(\sX,Q)$ if there are less than $\EDeg_Q(\sX)$ ED degree critical points relative to $u$.

\begin{definition}\label{def: generic data point}
A data point $u\in V$ is called \emph{generic} if it lies in the complement of $\Sigma(\sX,Q)$.
\end{definition}

\begin{remark}\label{rmk: nondegenerate bilinear forms}
In the above definition of Euclidean Distance degree, $Q$ can be any nondegenerate (not necessarily positive definite or even real) symmetric bilinear form.
In the following, by $S^2 V$ and $S^2 V^\mR$ we denote the space of all symmetric bilinear forms on $V$ and, respectively, on $V^\mR$.
Furthermore, we denote by $\sU\subset S^2V$ and $\sU^\mR\subset S^2 V^\mR$ the open dense subsets of nondegenerate symmetric bilinear forms on $V$ and, respectively, on $V^\mR$.
Finally, (real) positive definite forms in $S^2V$ and $S^2 V^\mR$ constitute a convex cone that is denoted by $S_+^2 V$ and $S^2_+ V^\mR$.
By its definition, $S^2_+ V^\mR$ consists of all inner products on $V^\mR$. 
\end{remark}

In the remainder of this subsection, we introduce a useful polynomial relation that describes the distance function \eqref{eq: distance function cone} and that depends polynomially on the coefficients of the bilinear form $Q$. Similarly to the definition of ED degree, this notion is presented for projective varieties but can be studied for affine varieties that are not necessarily cones. Let $\sX\subset\PP(V)=\PP^N$ be a projective variety, and consider its affine cone $C\sX\subset V$. Attached to the ED correspondence $\sE(\sX,Q)$ defined in \eqref{eq: ED correspondence} is the {\em extended ED correspondence} of $\sX$
\begin{equation}\label{eq: extended ED correspondence any X}
 \sE(\sX)\coloneqq\overline{\left\{(x,u,Q)\in V\times V\times S^2V\mid\text{$x\in(C\sX)_\mathrm{sm}$, $Q\in\sU$, and $u-x\in N_xC\sX$}\right\}}\,,
\end{equation}
where $N_xC\sX$ defined in \eqref{eq: def normal space} depends on $Q$ and $\sU$ is defined in Remark \ref{rmk: nondegenerate bilinear forms}. Furthermore, we define the {\em extended $\varepsilon$-offset correspondence} of $\sX$
\begin{equation}
 \sO\sE_{\varepsilon}(X)\coloneqq\sE(\sX)\cap\{(x,u,Q)\in V\times V\times S^2V \mid q(u-x)=Q(u-x,u-x)=\varepsilon^2\}\,,
\end{equation}
for all $\varepsilon\in\C$. Observe that $\sE(\sX)$ is not a subvariety of the hypersurface $\V(q(u-x)-\varepsilon^2)\subset V\times V\times S^2V$, otherwise we would have that $q(u-x)=\varepsilon^2$ for all $(x,u,Q)\in\sE(\sX)$. This implies that $\sO\sE_{\varepsilon}(X)$ is a subvariety of $\sE(\sX)$ of codimension 1.
The projection of $\sE(\sX)$ onto the last two factors $V\times S^2 V$ has generic zero-dimensional fibers. This implies that the Zariski closure of the projection of $\sO\sE_{\varepsilon}(X)$ onto $V\times S^2 V$ is a hypersurface. We define this variety as the {\em extended $\varepsilon$-offset hypersurface} of $\sX$ and we denote it by $\sO_{\varepsilon}(X)$. If we specialize both $\sO\sE_{\varepsilon}(X)$ and $\sO_{\varepsilon}(X)$ to a certain value of $Q$, we obtain respectively the {\em $\varepsilon$-offset correspondence} and the {\em $\varepsilon$-offset hypersurface} of $\sX$, that were introduced and studied first in \cite{horobet2019offset}. The main definition is the following.

\begin{definition}\label{def: extended ED polynomial}
The {\em extended ED polynomial} of $\sX$ is, up to a scalar factor, the unique generator of the ideal of $\sO_{\varepsilon}(X)$. We denote it by $\mathrm{EDpoly}_{\sX}(u,Q,\varepsilon^2)$.
\end{definition}

In the following, we consider $\mathrm{EDpoly}_{\sX}(u,Q,\varepsilon^2)$ as an element of the polynomial ring $\C[u,Q][\varepsilon]$, that is, we write it in the form
\begin{equation}\label{eq: write extended ED polynomial explicitly}
\mathrm{EDpoly}_{\sX}(u,Q,\varepsilon^2) = \sum_{i=0}^{D}p_i(u,Q)\,\varepsilon^{2i}\,,
\end{equation}
for some integer $D\ge 0$, where the coefficients $p_i(u,Q)$ are polynomials depending on the variables of $u$ and on the variables of the symmetric matrix $M_Q$ representing $Q$.
Since $\sX$ and the isotropic quadric $\sQ$ are defined by homogeneous ideals, the polynomial $\mathrm{EDpoly}_{\sX}(u,Q,\varepsilon^2)$ is homogeneous in $(u,Q,\varepsilon^2)$, in particular the value $\deg(p_i(u,Q))+2i$ does not depend on $i$.
If we fix $Q$, then $\mathrm{EDpoly}_{\sX}(u,Q,\varepsilon^2)$ specializes to the {\em ED polynomial} of $\sX$ introduced in \cite{ottaviani2020distance}. The following is an immediate consequence of \cite[Proposition 2.3]{ottaviani2020distance}.

\begin{proposition}\label{pro: roots ED polynomial}
The roots $\varepsilon^2$ of $\mathrm{EDpoly}_{\sX}(u,Q,\varepsilon^2)$ are precisely of the form $\varepsilon^2=q(u-x)$, where $x$ is a critical point of the polynomial function $x\mapsto q(u-x)$ on $\sX_{\mathrm{sm}}$. In particular, when restricting to positive definite symmetric bilinear forms $Q$, the distance $\varepsilon$ from the real locus $X^\mR$ to a data point $u\in V^\mR$ is a root of $\mathrm{EDpoly}_{\sX}(u,Q,\varepsilon^2)$. Furthermore, $\mathrm{EDpoly}_{\sX}(u,Q,0)=0$ for all $u\in X^\mR$.
\end{proposition}

A first result follows immediately from \cite[Theorem 2.9]{horobet2019offset}.

\begin{proposition}\label{prop: epsilon-degree ED polynomial}
For any fixed nondegenerate symmetric bilinear form $\widetilde{Q}$, the extended ED polynomial of $\sX$ evaluated at $\widetilde{Q}$ has the form
\begin{equation}
 \mathrm{EDpoly}_{\sX}(u,\widetilde{Q},\varepsilon^2) = \sum_{i=0}^{\EDeg_{\widetilde{Q}}(\sX)}p_i(u,\widetilde{Q})\,\varepsilon^{2i}\,. 
\end{equation}
Moreover, $p_{\,\EDeg_{\widetilde{Q}}(\sX)}(u,\widetilde{Q})\in \C[u]$ is not the zero polynomial.
\end{proposition}

In Section \ref{sec: ED polynomial} we investigate varieties defined by vanishing of the coefficients $p_i(u,Q)$ of \eqref{eq: write extended ED polynomial explicitly} and relate them to Conjecture \ref{conj: main}.

\subsection{Chern classes, duality, and polar classes}

We briefly outline some important notions of intersection theory and projective duality used in this paper.
Let $\sX\subset\PP^N$ be an irreducible projective variety. The $i$-th Chow group of $\sX$ is denoted by $A^i(\sX)$ and is the group of formal linear combinations with integer coefficients of $i$-codimensional irreducible subvarieties of $\sX$, up to rational equivalence. The direct sum $A^*(\sX)=\bigoplus_{i\ge 0}A^i(\sX)$ equipped with the intersection product is the Chow ring (or intersection ring) of $\sX$. We refer to \cite[Chapter 8]{fulton1998intersection} for more details.

\begin{example}\label{ex: Chow ring Segre product}
If $\sX=\PP^N\times\PP^N=\PP(V)\times\PP(V)$ embedded in $\PP(V\otimes V)\cong\PP^{N(N+2)}$, we have that $A^*(\PP^N\times\PP^N)\cong\Z[s,t]/(s^{N+1},t^{N+1})$, where $s$ and $t$ are the pullbacks of the hyperplane classes in the two factors of $\PP^N\times\PP^N$ via the two projection maps.\hfill$\diamondsuit$
\end{example}

Now let $\sE\to\sX$ be a vector bundle on $\sX$ of rank $r$.
For every $i\in\{0,\ldots,r\}$, one can associate to $\sE$ a rational equivalence class $c_i(\sE)\in A^i(\sX)$, called the {\em $i$-th Chern class} of $\sE$. We briefly recall some basic facts about Chern classes, and we refer to \cite[\S3.2]{fulton1998intersection} for more details. If $r=n=\dim(\sX)$, then $c_n(\sE)\in A^n(\sX)$ is a zero-dimensional cycle or a finite sum of classes of points. The sum of the coefficients of this cycle is an integer and is called the {\em top Chern number of $\sE$}. The {\em total Chern class} of $\sE$ is $c(\sE)\coloneqq\sum_{i=0}^rc_i(\sE)$. We recall two formal properties of Chern classes, which are useful tools to compute Chern classes of more complex vector bundles on $\sX$. First, given a Cartier divisor $D$ on $\sX$ and the line bundle $\sL=\sO(D)$, then $c(\sL)=1+D$, or equivalently $c_1(\sL)=D$. Secondly, given a short exact sequence $0\to\sA\to\sB\to\sC\to 0$ of vector bundles on $\sX$, the {\em Whitney sum} property $c(\sA)\cdot c(\sC)=c(\sB)$ holds.

We recall the notion of the dual variety of a projective variety, see \cite[Chapter 1]{GKZ}.

\begin{definition}\label{def: dual varieties}
Let $\sX\subset\PP^N$ be an irreducible projective variety.
A hyperplane $H\subset\PP^N$ is {\em tangent to $\sX$} if $H$ contains the projective tangent space $T_xX$ for some smooth point $x$ of $\sX$.
The {\it dual variety} $\sX^{\vee}$ of $\sX$ is the Zariski closure in $(\PP^N)^*\coloneqq\PP(V^*)$ of the set of all hyperplanes tangent to $\sX$.
A variety $\sX$ is {\it dual defective} if $\codim(\sX^{\vee})>1$. Otherwise, it is {\it dual nondefective}. When $\sX = \PP^N$, then $\sX^{\vee} = \emptyset$ and $\codim(\sX^{\vee}) = N+1$. 
\end{definition}

With the help of a nondegenerate symmetric bilinear form $Q$, one can identify the complex vector space $V$ with its dual $V^*$ via the isomorphism sending a vector $v^*\in V$ to the linear form $u\mapsto Q(v^*, u)$ on $V$. In this way, we identify the dual variety $\sX^{\vee}\subset(\PP^N)^*$ with the variety
\begin{equation}\label{eq: other notion of dual variety}
\sX^{\vee}=\overline{\bigcup_{x\in C\sX_{\mathrm{sm}}}\PP(N_xC\sX)}\subset\PP^N\,,
\end{equation}
where the bar stands for Zariski closure in $\PP^N$.
The variety at the right-hand side of \eqref{eq: other notion of dual variety} is also denoted by $\sX^\perp$, but we write $\sX^{\vee}$ for simplicity of notations. On one hand, this alternative notion of dual variety is useful because it allows us to consider $\sX$ and $\sX^\vee$ in the same space $\PP^N$. On the other hand, identity \eqref{eq: other notion of dual variety} depends on the choice of $Q$, in contrast with the classical notion of dual variety in Definition \ref{def: dual varieties}.

\begin{definition}\label{def: conormal variety}
Let $\sX\subset\PP^N$ be an irreducible projective variety. The {\it conormal variety} of $\sX$ is the incidence correspondence
\begin{equation}\label{eq: def conormal variety}
\sN_\sX\coloneqq\overline{\{(x,H)\in \PP^N\times(\PP^N)^*\mid\text{$x\in\sX_{\mathrm{sm}}$ and $H\supset T_x\sX$}\}}\,.
\end{equation}
\end{definition}
The projection of $\sN_\sX$ to the first factor $\PP^N$ defines a projective bundle over $\sX$ of rank equal to $\codim(X)-1$. This implies that $\sN_\sX$ is an irreducible variety of dimension $N-1$ in $\PP^N\times(\PP^N)^*$.
The projection of $\sN_\sX$ to $(\PP^N)^*$ is the dual variety $\sX^\vee$ of $\sX$.
A fundamental feature of conormal variety is the {\it biduality theorem} \cite[Chapter 1]{GKZ}: one has $\sN_\sX=\sN_{\sX^{\vee}}$. The latter implies $(\sX^{\vee})^\vee = X$, the so-called {\it biduality}.

Similarly, as in \eqref{eq: other notion of dual variety} for the dual variety, in the following, we identify the conormal variety $\sN_\sX$ with the variety in $\PP^N\times\PP^N$
\begin{equation}\label{eq: identity conormal}
\sN_\sX=\overline{\{(x,y)\in \PP^N\times\PP^N\mid\text{$x=[v]\in\sX_{\mathrm{sm}}$ and $y\in\PP(N_vC\sX)$}\}}\,,
\end{equation}
where $N_xC\sX$ is defined in \eqref{eq: def normal space}.
Since $\codim(\sN_\sX)=N+1$ within $\PP^N\times\PP^N$, using Example \ref{ex: Chow ring Segre product}, the rational equivalence class of $\sN_\sX$ in $A^{N+1}(\PP^N\times\PP^N)$ can be written in the form
\[
[\sN_\sX]=\delta_0(\sX)s^Nt+\delta_1(\sX)s^{N-1}t^2+\cdots+\delta_{N-1}(\sX)st^N\,.
\]
The coefficients $\delta_i(\sX)$ are known as the {\em multidegrees of $\sN_\sX$}. Letting $n=\dim(\sX)$, we have the relations (see \cite[Prop. (3), p.187]{kleiman1986tangency})
\begin{equation}\label{eq: identity multidegrees conormal polar degrees}
\delta_i(\sX)=0\quad\forall\,i>n\,,\quad\delta_i(\sX)=\deg(p_{n-i}(\sX))\quad\forall\,i\in\{0,\ldots,n\}\,,
\end{equation}
where $p_j(\sX)$ is the {\em $i$-th polar class} of $\sX$ \cite{piene1978polar}.

If we assume $\sX$ smooth, the invariants $\delta_i(\sX)$ may be computed utilizing the Chern classes $c_i(\sX)$ of $\sX$, namely the Chern classes of the tangent bundle $\sT_\sX\to\sX$. 
Suppose that the embedding $\sX\subset\PP^N$ is given by a line bundle $\sL$, and let $h=c_1(\sL)$.
One computes \cite[\S3]{Holme}:
\begin{equation}\label{eq: Holme}
\delta_i(\sX)=\sum_{j=0}^{n-i}(-1)^{j}\binom{n+1-j}{i+1}c_j(\sX)\cdot h^{n-j}\,,
\end{equation}
where $c_j(\sX)\cdot h^{n-j}$ is equal to $\deg(c_j(\sX))$ times the zero-dimensional class $h^n\in A^n(\sX)$. 
The right-hand side of \eqref{eq: Holme} is always a nonnegative integer. The integer $\mathrm{def}(X)=\codim(\sX^{\vee})-1$ equals the minimum $i$ such that $\delta_i(\sX)\neq 0$. Whenever $\sX$ is dual nondefective, one has
\begin{equation}\label{eq: degdual}
\mathrm{deg}(\sX^{\vee})=\delta_0(\sX)=\sum_{j=0}^n(-1)^j(n+1-j)\,c_j(\sX)\cdot h^{n-j}\,.
\end{equation}
We can also invert the relation between polar classes and Chern classes and get the new relation
\begin{equation}\label{eq: Holme inverted}
c_i(\sX)=\sum_{j=0}^i(-1)^{j}\binom{n+1-j}{i-j}\delta_{n-j}(\sX)\cdot h^{i-j}\,.
\end{equation}
If $\sX$ is a singular variety, then by replacing the Chern classes $c_i(\sX)$ by \emph{the Chern-Mather classes} $c_i^M(\sX)$, one obtains the same relations \eqref{eq: Holme} and \eqref{eq: Holme inverted} to polar classes of $\sX$, see \cite{piene1978polar}, \cite[Theorem 3]{piene1988cycles}, and \cite[Proposition 3.13]{aluffi2018projective} for more details.

\subsection{Generic ED degree and ED defect}\label{sec: generic ED degree and ED defect}

Given a variety $\sX\subset\PP^N$, the value of $\EDeg_Q(\sX)$ depends on the way how $\sX$ intersects the quadric $\sQ\subset \PP^N$ associated with the bilinear form $Q\in S^2_+(V)$. If $\sQ$ is in general position with respect to $\sX$, the ED degree of $\sX$ reaches its maximum possible value. This is the key point of the next result.

\begin{theorem}\cite[Theorem 5.4]{DHOST}\label{thm: ED degree sum polar classes}
Let $\sX\subset\PP^N$ be a projective variety and consider its conormal variety $\sN_\sX\subset\PP^N\times\PP^N$. If $\sN_\sX$ is disjoint from the diagonal $\Delta(\PP^N)$ of $\PP^N\times\PP^N$, then
\begin{equation}\label{eq: sum polar classes}
\EDeg_Q(\sX) = \delta_0(\sX)+\cdots+\delta_{N-1}(\sX) = \EDeg_Q(\sX^{\vee})\,.
\end{equation}
\end{theorem}

To satisfy the hypotheses of Theorem \ref{thm: ED degree sum polar classes}, it is sufficient that the varieties $\sX$ and $\sQ$ are transverse, namely that $\sX\cap\sQ$ is smooth and disjoint from $\sX_{\mathrm{sing}}$. In particular, this condition holds when $\sQ$ is generic in the projective space $\PP(S^2V)$ of quadric hypersurfaces.
Hence, the value in \eqref{eq: sum polar classes} is called {\em generic ED degree} of $\sX$ and is denoted by $\gEDeg(\sX)$.

Using the identities in \eqref{eq: Holme} and Theorem \ref{thm: ED degree sum polar classes}, one gets the generic ED degree of a smooth projective variety as a function of its Chern classes.

\begin{theorem}\cite[Theorem 5.8]{DHOST}\label{thm: ED degree Chern classes}
Consider the assumptions of Theorem \ref{thm: ED degree sum polar classes}. Furthermore, we assume that $\sX\subset\PP^N$ is smooth of dimension $n$. Then
\begin{equation}\label{eq: ED degree Chern classes}
\EDeg_Q(\sX) = \gEDeg(\sX) = \sum_{i=0}^n(-1)^i(2^{n+1-i}-1)\deg(c_i(\sX))\,.
\end{equation}
\end{theorem}
The previous result has been extended to singular varieties by Aluffi in \cite[Proposition 2.9]{aluffi2018projective}.

For a specific choice of a bilinear form $Q$, the value in \eqref{eq: sum polar classes} is only an upper bound for the ED degree of a variety $\sX\subset\PP^N$ with respect to $Q$. This motivates the definition of {\em defect of ED degree} of $\sX$ with respect to $Q$ as the difference
\begin{equation}\label{eq: defect of ED degree}
\EDef_Q(\sX)\coloneqq \gEDeg(\sX)-\EDeg_Q(\sX)\,.
\end{equation}
The study of this invariant is the main content of \cite{maxim2020defect}. Let us recall their main results for the reader's convenience. Below, $\mathscr{X}_0$ denotes a Whitney stratification of the singular locus of $\sX\cap\sQ$. The order between strata $V<S$ means $V\subset\overline{S}$, and $L_{V,S}$ is the complex link of the pair of strata $(V,S)$, see \cite[Chapter 2]{GorMac-Morse}. Finally, 
\[
\mu_V \coloneqq \chi(\widetilde{H}^*(F_V;\Q))
\]
is the Euler characteristic (in reduced cohomology) of the Milnor fiber $F_V$ associated with $V$ around any point $x\in V$. If $V = \{x\}$ is an isolated singularity, let $\{ f=0\}$ be a local equation for $\sX\cap\sQ$ in a neighborhood of $x$; let $m =\dim_{\C} (\sX\cap\sQ)$. Milnor proved \cite[Chapter 7]{milnor-book-sings} that $\dim_{\C}H^j(F_V, \Q) = 0$ for $j\neq 0, m$, while $\dim_{\C}H^0(F_V, \Q) = 1$ and $\dim_{\C}H^{m}(F_V, \Q) = \mu$, where 
\[
\mu = \dim_{\C}\frac{\C[[ x_1, \dots, x_{m+1} ]]}{ \left( \frac{\partial f}{\partial x_1}, \dots , \frac{\partial f}{\partial x_{m+1}}\right)}
\]
is what we call nowadays the \emph{Milnor number}. The relation with $\mu_V$ is then clear: $\mu_V = (-1)^m\mu$. The following results will be useful in the sequel.

\begin{theorem}\cite[Theorem~1.5]{maxim2020defect}\label{thm: ED defect general}
Let $\sX\subset\PP^N$ be a smooth irreducible projective variety not contained in the isotropic quadric $\sQ$. Then
\begin{equation}\label{eq: ED defect main formula}
\EDef_Q(\sX) = \sum_{V\in \mathscr{X}_0} (-1)^{\codim_{\sX\cap\sQ}(V)}\alpha_V \cdot \gEDeg(\bar{V})
\end{equation}
with $\alpha_V = \mu_V -\sum_{V<S}\chi_c(L_{V,S})\cdot \mu_S$, where $\chi_c$ denotes the Euler characteristic for the homology with compact support. 
\end{theorem}

The previous formula \eqref{eq: ED defect main formula} is generally difficult to use. However, in our computations, we mostly apply the next two nice simplifications.

\begin{corollary}\cite[Corollary~1.8]{maxim2020defect}\label{cor: ED defect isolated singularities}
Assume that $\sZ=\mathrm{Sing}(\sX\cap\sQ)$ consists of isolated points. Then
\begin{equation}
 \EDef_Q(\sX) = \sum_{x\in\sZ}\mu_x\,,
\end{equation}
where $\mu_x$ is the Milnor number of the isolated singularity $x\in\sZ$.
\end{corollary}

\begin{corollary}\cite[Corollary~1.9]{maxim2020defect}\label{cor: ED defect equisingular}
Assume that $\sZ=\mathrm{Sing}(\sX\cap\sQ)$ is connected and that it is a closed (smooth
and connected) stratum in a Whitney stratification of $\sX\cap\sQ$ (namely, $\sX\cap\sQ$ is equisingular along $Z$). Then
\begin{equation}
 \EDef_Q(\sX) = \mu\cdot\gEDeg(\sZ)\,,
\end{equation}
where $\mu$ is the Milnor number of the isolated transversal singularity at some point of $Z$.
\end{corollary}

\subsection{Morse theory}\label{sec: Morse theory}

In this subsection, we briefly recall basic facts in Morse Theory. Given a differentiable manifold $\sM$ and a differentiable function $f\colon\sM\to\R$, a point $p\in\sM$ is called a {\em critical point of $f$} if the differential
\[
\dd_p f\colon T_p\sM \to T_{f(p)} \R \cong \R
\]
vanishes. For $v, w\in T_p\sM$, we can extend $v, w$ to vector fields $\tilde{v}, \tilde{w}$ on an open neighborhood of $p$. The Hessian $H_p(f)$ of $f$ at a critical point $p$ is the symmetric bilinear map defined by
\[
H_p(f)\colon T_p\sM \times T_p\sM \to\R, \quad H_p(f)(v,w) = \tilde{v} \bigl(\tilde{w}(f)\bigr).
\]
\begin{definition}\label{def: index, Morse function}
Let $f\colon\sM\to\R$ be a differentiable function over a differentiable manifold $\sM$, and $p$ be a critical point of $f$. The {\em index} of $p$, denoted by $\lambda_p$, is the dimension of the subspace of $T_p\sM$ on which the Hessian $H_p(f)$ is negative definite. The critical point $p$ is said to be {\em nondegenerate} if and only if $H_p(f)$ is nondegenerate. Moreover, we say that $f$ is a {\em Morse function} if the critical points of $f$ are all nondegenerate.
\end{definition}
Note that nondegenerate critical points are isolated~\cite[Corollary~2.3]{Milnor1963}. Hence, a Morse function $f\colon\sM\to\R$ on a compact manifold $\sM$ has finitely many critical points. Furthermore, if we let $m_k$ be the number of critical points of $f$ of index $k$ for all $k\in\{0,\ldots,\dim(\sM)\}$, then
\begin{equation}\label{eq: inequalities Betti}
m_k \ge b_k(\sM)
\end{equation}
for all $k$, where $b_k(\sM)$ is the $k$-th Betti number of $\sM$, i.e., the rank of the $k$-th homology group $H_k(\sM,\Q)$. These inequalities are known as weak {\em Morse inequalities}. In this paper, we also need the following strong Morse inequalities, see for example \cite[\S $5$]{Milnor1963} or \cite[Theorem~3.33]{BH04}:
\begin{theorem}\label{thm: strong Morse}
For a Morse function $f\colon\sM\to\R$ on a compact manifold $\sM$, we have
\begin{enumerate}
\item 
$\sum_{k=0}^i (-1)^{i-k} m_k \ge \sum_{k=0}^i (-1)^{i-k} b_k(\sM)$ for every $i\in\{0,\ldots,\dim(\sM)\}$.
\item
$\sum_{k=0}^{\dim(\sM)} (-1)^k m_k = \sum_{k=0}^{\dim(\sM)} (-1)^k b_k(\sM)$.
\end{enumerate}
\end{theorem}

In the following lemma, we study a necessary and sufficient condition under which the distance function from a point to a real differentiable manifold is Morse.
First, given a differentiable submanifold $\sM\subset V^\mR$, similarly as in \eqref{eq: ED correspondence} we define the {\em (real) ED correspondence} of $\sM$ as
\begin{equation}\label{eq: real ED correspondence manifold M}
\sE(\sM,Q)^\mR\coloneqq\{(x,u)\in V^\mR\times V^\mR\mid\text{$x\in \sM$ and $u-x\in N_x \sM$}\}\,,
\end{equation}
where $N_x\sM=\left\{v^*\in V^\mR \mid Q(v^*,v)=0,\ v\in T_x \sM\right\}$ is the normal space to $\sM$ at $x\in \sM$, cf. \eqref{eq: def normal space}.

\begin{lemma}\label{lem: generic=Morse}
Let $\sM$ be a differentiable submanifold of a real vector space $V^\mR$ and let $Q\colon V^\mR\times V^\mR\to\R$ be an inner product on $V^\mR$ with the associated quadratic form $q$ defined as $q(x)=Q(x,x)$, $x\in V^\mR$. 
For a given $u\in V^\mR$, the distance function
\begin{equation}\label{eq: dist^Q}
\dist^Q_{\sM,u} \colon \sM\to\R\,,\quad x\mapsto q(u-x)
\end{equation}
is Morse if and only if $u$ is a regular value of the projection $\pr_2\colon\sE(\sM,Q)^\mR\to V^\mR$ on the second factor.
In particular, for a dense subset of $u \in V^\mR$, the function \eqref{eq: dist^Q} is Morse.
\end{lemma}
\begin{proof}
We base our arguments on the proof of \cite[Lemma 4.6]{GG:stablemaps}. 
Let us consider the function
\begin{equation}\label{eq: def F(x,u) differential}
F\colon \sM\times V^\mR \to T^*\sM\,,\quad F(x,u)\coloneqq\left(x, -\frac{1}{2}\dd_x\dist^Q_{\sM,u}\right)\,,
\end{equation}
where $\dd_x\dist^Q_{\sM,u}$ is the differential of $\dist^Q_{\sM,u}$ and
\begin{equation}\label{eq: identity differential F(x,u)}
\dd_x\dist^Q_{\sM,u}(\dot{x})=-2\,Q(u-x,\dot{x})\quad\forall\,\dot{x}\in T_x\sM\,.
\end{equation}
By a standard fact from transversality theory, the function $\dist^Q_{\sM,u}$ is Morse if and only if $F_u \coloneqq F(\cdot,u)$ is transverse to the zero section $\sZ\coloneqq\{(x,0)\mid x\in \sM\} \subset T^*\sM$.
We claim that the map $F$ (as a whole) is transverse to $\sZ$. 
For this, take any $(x,u)\in \sM\times V^\mR$ such that $F(x,u)\in\sZ$.
We need to show that 
\begin{align}\label{eq: identity F transverse to Z}
\dd_{(x,u)} F(T_x\sM\oplus T_uV^\mR) + T_{F(x,u)}\sZ=T_{F(x,u)}T^*\sM\,.
\end{align}
Note that $F(x,u) = (x,0)$ and $T_{F(x,u)}T^*\sM=T_{(x,0)}T^*\sM = T_{(x,0)}\sZ \oplus T^*_x\sM$. On the other hand, the differential of $F_x\coloneqq F(x,\_)$ is given as 
\[
\dd_u F_x\colon T_u V^\mR\to T_x^*\sM\,,\quad \dot{u}\longmapsto\left[\dot{x} \mapsto Q(\dot{u},\dot{x})\right].
\]
Since any $f\in T^*_x\sM$ is obtained as the restriction of a linear form on $V^\mR$ to $T_x\sM\subset V^\mR$, it can be written as $f(\dot{x}) = Q(\dot{u},\dot{x})$, $\dot{x}\in T_x\sM$, for some $\dot{u}\in V^\mR=T_uV^\mR$. Therefore, $\dd_u F_x$ is surjective. This implies that $F$ is transverse to $\sZ$ and, in particular, the real ED correspondence $\sE(\sM,Q)^\mR=F^{-1}(\sZ)$ is a differentiable submanifold of $\sM\times V^\mR$ of dimension $\dim(V^\mR)$.

Let now $u\in V^\mR$ be a regular value of $\pr_2\colon \sE(\sM,Q)^\mR\to V^\mR$ and let $x\in\sM$ be any point with $(x,u)\in \sE(\sM,Q)^\mR$. Then the differential $\dd_{(x,u)}\pr_2\colon T_{(x,u)}\sE(\sM,Q)^\mR\to T_u V^\mR$ defined by $(\dot{x},\dot{u})\mapsto \dot{u}$ is an isomorphism, hence
\begin{equation}\label{eq: direct sum tangent spaces}
T_x\sM\oplus T_uV^\mR=T_x \sM\oplus T_{(x,u)} \sE(\sM,Q)^\mR\,. 
\end{equation}
Applying $\dd_{(x,u)}F$ to both sides of \eqref{eq: direct sum tangent spaces} yields $\dd_{(x,u)} F(T_x\sM\oplus T_uV^\mR)=\dd_{x} F_u(T_x\sM)+ T_{F(u,x)}\sZ$.
Combining this with \eqref{eq: identity F transverse to Z} we obtain $\dd_x F_u(T_x\sM) + T_{F_u(x)}\sZ= T_{F_u(x)} T^*\sM$, which means that $F_u$ is transverse to $\sZ$ or, equivalently, $\dist^Q_{\sM,u}$ is Morse.

Vice versa, if $(x,u)\in \sE(\sM,Q)^\mR$ is a critical point of $\pr_2$, the kernel of $\dd_{(x,u)}\pr_2$ contains a nontrivial vector $(\dot{x},0)\in T_{(x,u)}\sE(\sM,Q)^\mR$.
Since $\dd_x F_u$ sends $(\dot{x},0)$ to $T_{F_u(x)}\sZ=T_{(x,0)}\sZ$ and since $\dim(T^* \sM)=\dim(\sM)+\dim(\sZ)$, the map $F_u$ cannot be transverse to $\sZ$ and $\dist^Q_{\sM,u}$ is not a Morse function. 

Finally, by a version of Sard's theorem for smooth maps (see \cite[Cor 1.14]{GG:stablemaps}), the function $\dist^Q_{\sM,u}$ is Morse for a dense subset of $u\in V^\mR$. 
\end{proof}

\section{ED degrees for Segre-Veronese varieties}\label{sec: product metrics tensor spaces}

In this section, we start by giving an explicit formula for the generic ED degree of the Segre-Veronese variety. Then we propose Conjecture~\ref{conj: main} which motivates our work. As evidence, we prove Theorem~\ref{thm: local minimality Frobenius ED degree} revealing the local minimality of the ED degree induced by the Frobenius inner product \eqref{eq: Frobenius inner product for tensors}.

Throughout the rest of the paper, $k$ is a positive integer, while $\bn=(n_1,\ldots,n_k)$ and $\bd=(d_1,\ldots,d_k)$ are $k$-tuples of nonnegative integers. Consider $k$ real vector spaces $V_1^\mR,\ldots,V_k^\mR$ with $\dim(V_i^\mR)=n_i+1$. Recall that $V_i=V_i^\mR\otimes\C$. Given a $k$-tuple $\bd$, we define $S^\bd V\coloneqq S^{d_1}V_1\otimes\cdots\otimes S^{d_k}V_k$. In particular, if $d_i=1$ for all $i$, then $S^\bd V=V\coloneqq V_1\otimes\cdots\otimes V_k$. In this section we call $N$ the dimension of the projective space $\PP(V)$, that is $\binom{n_1+d_1}{d_1}\cdots\binom{n_k+d_k}{d_k}-1$, so we write $\PP(V)=\PP^N$.

Let $\PP^\bn$ be the Cartesian product $\PP(V_1)\times\cdots\times\PP(V_k)=\PP^{n_1}\times\cdots\times\PP^{n_k}$.
We denote by {$\nu_\bd\colon\PP^\bn\to\PP(S^\bd V)$} the Segre-Veronese embedding of $\PP^\bn$ via the line bundle $\sO_{\PP^\bn}(\bd)$. Its image $\sV_{\bd,\bn}=\nu_\bd(\PP^\bn)$ is the {\em Segre-Veronese variety} of $\PP(S^\bd V)$.
More precisely, $\sV_{\bd,\bn}$ is populated by the so-called {\em decomposable tensors} of $\PP(S^\bd V)$, which are tensors of the form $\nu_\bd(x_1,\ldots,x_k)=x_1^{d_1}\otimes\cdots\otimes x_k^{d_k}$, where $x_i\in V_i$ for all $i\in[k]$.
When $k=1$, $V_1=V$, $\bn=(n)$ and $\bd=(d)$ we just write $\sO_{\PP^\bn}(\bd)=\sO_{\PP^n}(d)$ and $\sV_{\bd,\bn}=\sV_{d,n}$ is the {\em Veronese variety} of $\PP(S^dV)$. When $d_i=1$ for all $i\in[k]$, we use the notation $\mathbf{1}=1^k=(1,\ldots,1)$ and the variety $\sV_{\mathbf{1},\bn}$, also denoted with $\Sigma_\bn$, is the {\em Segre variety} of $\PP(V)$. If also $\bn=\mathbf{1}$, then we use the notation $\Sigma_k$ to denote $\Sigma_{\mathbf{1}}$.

Our goal is to study the ED degrees of the varieties $\sV_{\bd,\bn}$, hence we need to consider a positive definite symmetric bilinear form on the space $\PP(S^\bd V)$ too. A natural choice is given in the following definition.

\begin{definition}\label{def: Frobenius inner product}
For all $i\in[k]$, let $Q_i$ be a positive definite symmetric bilinear form on the space $V_i^\mR$, with associated quadratic form $q_i$.
Consider two decomposable tensors $T = x_1^{d_1}\otimes\cdots\otimes x_k^{d_k}$ and $T' = y_1^{d_1}\otimes\cdots\otimes y_k^{d_k}$ on $S^\bd V^\mR$.
The {\em Frobenius inner product} between $T$ and $T'$ is
\begin{equation}\label{eq: Frobenius inner product for tensors}
Q_F(T,T')\coloneqq \prod_{i=1}^k Q_i(x_i, y_i)\,,
\end{equation}
and it is extended to every tensor in $S^\bd V$ by linearity. The isotropic quadrics in $\PP(V_i)=\PP^{n_i}$ are denoted by $\sQ_i$, while the corresponding isotropic quadric in $\PP(S^\bd V)$ is denoted by $\sQ_F$.
\end{definition}

The ED degree of a Segre-Veronese variety $\sV_{\bd,\bn}$ with respect to a Frobenius inner product $Q_F$ is the content of the next result.

\begin{theorem}{\cite[Theorem 1]{FO}}\label{thm: FO formula}
The ED degree of the Segre-Veronese variety $\sV_{\bd,\bn}\subset\PP(S^\bd V)$ with respect to a Frobenius inner product $Q_F$ equals the coefficient of the monomial $h_1^{n_1}\cdots h_k^{n_k}$ in the expansion of
\[
\prod_{i=1}^k\frac{\widehat{h}_i^{n_i+1}-h_i^{n_i+1}}{\widehat{h}_i-h_i},\quad\widehat{h}_i\coloneqq\left(\sum_{j=1}^k d_jh_j\right)-t_i\,.
\]
\end{theorem}

Interestingly, Aluffi and Harris computed the same ED degree with an independent formula.

\begin{theorem}{\cite[\S9]{aluffi-harris}}\label{thm: AH formula}
The ED degree of the Segre-Veronese variety $\sV_{\bd,\bn}\subset\PP(S^\bd V)$ with respect to a Frobenius inner product $Q_F$ equals the coefficient of the monomial $h_1^{n_1}\cdots h_k^{n_k}$ in
\[
\frac{1}{1-d_1h_1-\cdots-d_kh_k}\cdot\prod_{i=1}^k\frac{(1-h_i)^{n_i+1}}{1-2d_ih_i}\,.
\]
\end{theorem}

Since $\sV_{\bd,\bn}$ is a smooth toric variety on $\PP(S^\bd V)$, the degrees of its Chern classes are relatively easy to compute. The proof of the next result relies on computations made by the fourth author of this paper in a more general setting in \cite[Chapter 5]{sodphd}.

\begin{proposition}\label{prop: degrees Chern classes Segre-Veronese}
The degrees of the Chern classes of the tangent bundle of the Segre-Veronese variety $\sV_{\bd,\bn}\subset\PP(S^\bd V)$ are
\begin{equation}\label{eq: degree ith Chern class Segre-Veronese}
\deg(c_i(\sV_{\bd,\bn})) = (|\bn|-i)!\sum_{|\balpha|=i}\gamma_{\balpha}\,\bd^{\bn-\balpha}\,,
\end{equation}
where $|\bn|=n_1+\cdots+n_k=\dim(\sV_{\bd,\bn})$, $\bd^{\bn-\balpha}=d_1^{n_1-\alpha_1}\cdots d_k^{n_k-\alpha_k}$ and for all $\alpha\in\Z_{\ge0}^k$
\begin{equation}\label{eq: def gamma alpha}
\gamma_{\balpha}\coloneqq
\begin{cases}
\prod_{i=1}^k\frac{\binom{n_i+1}{\alpha_i}}{(n_i-\alpha_i)!} & \mbox{if $n_i\ge\alpha_i\ \forall\,i\in[k]$}\\
0 & \mbox{otherwise.}
\end{cases}
\end{equation}
\end{proposition}
\begin{proof}
Recall that $\sV_{\bd,\bn}$ is the image of the embedding of $\PP^\bn$ via the line bundle $\sO_{\PP^\bn}(\bd)$. In particular, if $h_i=c_1(\sO_{\PP^{n_i}}(1))$ for all $i\in[k]$, then by \cite[Thm. 14.10]{milnor1974characteristic},
\[
c(\PP^{n_i}) = (1+h_i)^{n_i+1} = \sum_{\ell_i=0}^{n_i}\binom{n_i+1}{\ell_i}h_i^{\ell_i}\quad\text{and}\quad c(\PP^\bn)=c(\PP^{n_1})\cdots c(\PP^{n_k})\,.
\]
Similarly, as in Example \ref{ex: Chow ring Segre product}, we have $A^*(\PP^\bn)\cong\Z[h_1,\ldots,h_k]/(h_1^{n_1+1},\ldots,h_k^{n_k+1})$, and
\begin{align*}
c(\PP^\bn) = \prod_{i=1}^k c(\PP^{n_i})=\prod_{i=1}^k\left(\sum_{\ell_i=0}^{n_i}\binom{n_i+1}{\ell_i}h_i^{\ell_i}\right) = \sum_{j=0}^{|\bn|}\sum_{|\balpha|=j}\left(\prod_{i=1}^k\binom{n_i+1}{\alpha_i}\right) h_1^{\alpha_1}\cdots h_k^{\alpha_k}\,,
\end{align*}
where $|\bn|=n_1+\cdots+n_k=\dim(\PP^\bn)=\dim(\sV_{\bd,\bn})$.
Therefore, the $i$-th Chern class of $\PP^\bn$ is
\begin{equation}\label{eq: Chern class product}
c_i(\PP^\bn)=\sum_{|\balpha|=i}\left(\prod_{i=1}^k\binom{n_i+1}{\alpha_i}\right) h_1^{\alpha_1}\cdots h_k^{\alpha_k}\quad\forall\,0\le i\le |\bn|\,.
\end{equation}
The hyperplane class of $\PP(S^\bd V)$ restricted to $\sV_{\bd,\bn}$ is precisely $c_1(\sO_{\PP^\bn}(\bd))=d_1h_1+\cdots+d_kh_k$. We consider its power
\begin{equation}\label{eq: power Chern class line bundle product}
c_1(\sO_{\PP^\bn}(\bd))^{|\bn|-i}=\sum_{|\bomega|=|\bn|-i}\binom{|\bn|-i}{\omega}(d_1h_1)^{\omega_1}\cdots (d_kh_k)^{\omega_k}\,,
\end{equation}
where $\binom{|\bn|-i}{\omega}\coloneqq\frac{(|\bn|-i)!}{\omega_1!\cdots\omega_k!}$.
The product between the two polynomials in \eqref{eq: Chern class product} and \eqref{eq: power Chern class line bundle product} equals an integer times the class of a point $h_1^{n_1}\cdots h_k^{n_k}\in A^{|\bn|}(\PP^\bn)$. Recalling the relations $h_1^{n_1+1}=\cdots=h_k^{n_k+1}=0$, for all $i$ and $\alpha,\omega$ such that $|\balpha|=i$ and $|\bomega|=|\bn|-i$, the monomial $h_1^{\alpha_1+\omega_1}\cdots h_k^{\alpha_k+\omega_k}$ is nonzero if and only if $\balpha+\bomega=\bn$. Therefore, the coefficient of $h_1^{n_1}\cdots h_k^{n_k}$ in the previous product is precisely the right-hand side of \eqref{eq: degree ith Chern class Segre-Veronese}. This concludes the proof.
\end{proof}

The next statement follows by combining Theorem \ref{thm: ED degree Chern classes} and Proposition \ref{prop: degrees Chern classes Segre-Veronese}.

\begin{theorem}\label{thm: gen ED degree Segre-Veronese}
The generic ED degree of the Segre-Veronese variety $\sV_{\bd,\bn}\subset\PP(S^\bd V)$ is
\begin{equation}\label{eq: gen ED degree Segre-Veronese}
\gEDeg(\sV_{\bd,\bn})=\sum_{i=0}^{|\bn|}(-1)^i(2^{|\bn|+1-i}-1)(|\bn|-i)!\sum_{|\balpha|=i}\gamma_{\balpha}\,\bd^{\bn-\balpha}\,,
\end{equation}
where the coefficients $\gamma_{\balpha}$ are defined in \eqref{eq: def gamma alpha}.
\end{theorem}

The formula \eqref{eq: gen ED degree Segre-Veronese} simplifies a lot for nonsymmetric tensors of binary format, namely for $\bn=\bd=\mathbf{1}$. We recall that a {\em derangement} of a set is a permutation of the elements of the set in which no element appears in its original position. The number of derangements of a set with $k$ elements is denoted by $!k$, and is also called the {\em subfactorial} of $k$. In particular $!0=1$, $!1=0$ and $!k=(k-1)[!(k-1)+!(k-2)]$ for all $k\ge 2$.
This recursion yields the following formulas.

\begin{lemma}\label{lem: identity derangements}
For all integers $k\ge 0$, we have the identity
\begin{equation}
    !k = k!\sum_{i=0}^k\frac{(-1)^i}{i!},\quad k\in \mathbb{N}.
\end{equation}
\end{lemma}


\begin{lemma}\label{lem: identity Gamma}
For all integers $k\ge 0$, we have the identity
\begin{equation}
 \frac{\Gamma(k+1,-2)}{e^2} = k!\sum_{i=0}^k\frac{(-1)^i}{i!}2^i\,,
\end{equation}
where $\Gamma(\alpha,x)\coloneqq\int_x^\infty t^{\alpha-1}e^{-t}dt$ is the {\em incomplete Gamma function}.
\end{lemma}

\begin{corollary}\label{corol: gen ED degree binary tensors}
If $\bn=\bd=\mathbf{1}=1^k$, then the generic ED degree of the Segre variety $\Sigma_k$ is
\begin{equation}\label{eq: gen ED degree binary tensors}
\gEDeg(\Sigma_k)=2^{k+1}\cdot!k-\frac{\Gamma(k+1,-2)}{e^2}\,.
\end{equation}
\end{corollary}
\begin{proof}
Consider the numbers $\gamma_{\balpha}$ defined in \eqref{eq: def gamma alpha}. Under our assumptions,
for a tuple $\balpha\in\Z_{\ge 0}^k$, we have that $\gamma_{\balpha}=\prod_{i=1}^k\binom{2}{\alpha_i}=2^{|\balpha|}$ if $\alpha_i\le 1$ for all $i\in[k]$, otherwise $\gamma_{\balpha}=0$. In this way, we can restrict only to summands $\alpha\in\{0,1\}^k$ and, by Theorem~\ref{thm: gen ED degree Segre-Veronese}, we conclude that
\begin{align*}
 \gEDeg(\Sigma_k) &= \sum_{i=0}^k(-1)^i(2^{k+1-i}-1)(k-i)!\sum_{|\balpha|=i}\gamma_{\balpha} = \sum_{i=0}^k(-1)^i(2^{k+1-i}-1)(k-i)!\sum_{\substack{\alpha\in\{0,1\}^k\\|\balpha|=i}}2^i\\
 &= \sum_{i=0}^k(-1)^i(2^{k+1-i}-1)(k-i)!\binom{k}{i}2^i = k!\sum_{i=0}^k\frac{(-1)^i}{i!}(2^{k+1}-2^i)\,.
\end{align*}
The equality in \eqref{eq: gen ED degree binary tensors} follows by Lemma \ref{lem: identity derangements} and Lemma \ref{lem: identity Gamma}.
\end{proof}

In general, the ED degree of $\sV_{\bd,\bn}$ with respect to a Frobenius inner product $Q_F$ is much smaller than its generic ED degree. For example, comparing Theorem \ref{thm: FO formula} and Corollary \ref{corol: gen ED degree binary tensors} for $\bn=\bd=\mathbf{1}=1^k$ yields the inequality
\begin{equation}
 \gEDeg(\Sigma_k)=2^{k+1}\cdot!k-\frac{\Gamma(k+1,-2)}{e^2}\ge k!=\EDeg_{Q_F}(\Sigma_k)\,,
\end{equation}
which is strict for all $k\ge 2$. We display the values of both sides for small $k$ in Table \ref{tab: generic ED degree Segre binary}.
\begin{table}[ht]
 \centering
 \begin{tabular}{c|c|c|c|c|c|c|c|c|c|c}
 $k$ & 1 & 2 & 3 & 4 & 5 & 6 & 7 & 8 & 9 & 10\\\hline
 $\gEDeg(\Sigma_k)$ & 1&6&34&280&2808&33808&473968&7588992&136650880&2733508864\\\hline
 $\EDeg_{Q_F}(\Sigma_k)$ &1&2&6&24&120&720&5040&40320&362880&3628800
 \end{tabular}
 \vspace{5pt}
 \caption{Comparison between generic ED degree and ED degree of $\Sigma_k$ for a Frobenius inner product for small values of $k$.}
 \label{tab: generic ED degree Segre binary}
\end{table}

Consider instead the case $\bd=(1,1)$, $\bn=(n_1,n_2)$ (assume $n_1\le n_2$), where $\Sigma_\bn$ parametrizes all $(n_1+1)\times(n_2+1)$ matrices of rank at most one. On the one hand, Theorem \ref{thm: FO formula} gives the value $\EDeg_{Q_F}(\Sigma_\bn)=n_1+1$, which agrees with the Eckart-Young theorem, see for example \cite{ottaviani2015geometric} for more details.
On the other hand, the formula in Theorem \ref{thm: gen ED degree Segre-Veronese} simplifies to
\begin{equation}\label{eq: gen ED degree rank-one matrices}
\gEDeg(\Sigma_{\bn})=\sum_{i=0}^{n_1+n_2}(-1)^i(2^{n_1+n_2+1-i}-1)(n_1+n_2-i)!\sum_{\substack{\alpha_1+\alpha_2=i\\n_i\ge\alpha_i}}\frac{\binom{n_1+1}{\alpha_1}\binom{n_2+1}{\alpha_2}}{(n_1-\alpha_1)!(n_2-\alpha_2)!}\,.
\end{equation}
Another formula for $\gEDeg(\Sigma_{\bn})$ for $\bn=(n_1,n_2)$ can be derived from \cite[Proposition 5.5]{zhang2018chern}, in particular $\Sigma_{\bn}$ corresponds to the determinantal variety $\tau_{n_2+1,n_1+1,n_1}$ therein.
We display in Table \ref{tab: generic ED degree Segre matrices} the first values of $\gEDeg(\Sigma_\bn)$.
\begin{table}[ht]
 \centering
 \begin{tabular}{c|c|c|c|c|c|c|c|c|c|c}
 $n_1\le n_2$&1&2&3&4&5&6&7&8&9&10\\\hline
 1&6&10&14&18&22&26&30&34&38&42\\\hline
 2&&39&83&143&219&311&419&543&683&839\\\hline
 3&&&284&676&1324&2292&3644&5444&7756&10644\\\hline
 4&&&&2205&5557&11821&22341&38717&62805&96717\\\hline
 5&&&&&17730&46222&104026&209766&388722&673854\\\hline
 6&&&&&&145635&388327&910171&1928191&3768211\\\hline
 7&&&&&&&1213560&3288712&7947416&17500200\\\hline
 8&&&&&&&&10218105&28031657&69374105\\\hline
 9&&&&&&&&&86717630&240186706\\\hline
 10&&&&&&&&&&740526303
 \end{tabular}
 \vspace{5pt}
 \caption{Generic ED degrees of $\Sigma_\bn$ for small values of $\bn=(n_1,n_2)$ with $n_1\le n_2$.}
 \label{tab: generic ED degree Segre matrices}
\end{table}

For example, the first row corresponds to the format $2\times (n_2+1)$, and as shown in \cite[Remark 4.21]{ottaviani2021asymptotics}, we have that
\begin{equation}
\gEDeg(\Sigma_{(1,n_2)})=4n_2+2>2=\EDeg_{Q_F}(\Sigma_{1,n_2})\,,
\end{equation}
in particular, the left-hand side diverges for $n_2\to\infty$, while the right-hand side remains constant.
With a similar computation, which we omit, one can also verify that
\begin{equation}
 \gEDeg(\Sigma_{(2,n_2)})=8n_2^2+4n_2-1\quad\forall\,n_2\ge 2\,.
\end{equation}
We also note that the diagonal of Table \ref{tab: generic ED degree Segre matrices} coincides with the integer sequence in \cite[\href{http://oeis.org/A231482}{A231482}]{oeis}, in the context of computing the number of nonlinear normal modes for a fully resonant Hamiltonian system, see \cite{vanstraten1989note}.

The following general conjecture motivates all the research done in this work.

\begin{conjecture}\label{conj: main}
Consider the Segre-Veronese variety $\sV_{\bd,\bn}\subset\PP(S^\bd V)$. Let $Q_F$ be a Frobenius inner product on $S^\bd V^\mR$. For any positive definite symmetric bilinear form $Q$ on $S^\bd V^\mR$, we have
\[
\EDeg_Q(\sV_{\bd,\bn})\ge\EDeg_{Q_F}(\sV_{\bd,\bn})\,,
\]
namely the maximum defect of ED degree of $\sV_{\bd,\bn}$ is given by the difference
\[
\gEDeg(\sV_{\bd,\bn})-\EDeg_{Q_F}(\sV_{\bd,\bn})\,.
\]
\end{conjecture}

Even though we do not have a proof of Conjecture \ref{conj: main} in general, in the next section we settle it in the case of matrices, that is, tensors of order two.
Specifically, in Theorems \ref{thm: lower bound Q-distance function} and \ref{thm: symmetric_matrices} we show that, given any inner product, the associated distance function from a generic real (symmetric) matrix to the manifold of rank-1 real (symmetric) matrices has at least as many critical points as the distance function associated to the Frobenius inner product.

In Section \ref{sec: ED defects} we perform a detailed study of ED defects of various varieties, with a particular focus on special Segre-Veronese varieties.
The results stated there are of independent interest, and they give further evidence that Conjecture \ref{conj: main} holds, see Corollaries \ref{cor: minimum ED degree rational curve} and \ref{cor:3-rd Veronese}.

In Remark \ref{rem: approach} below we outline a possible approach for attacking Conjecture \ref{conj: main}. We plan to investigate it in forthcoming works.
In the rest of this section, we show a ``local'' version of Conjecture \ref{conj: main}, which follows from a more general property of ED degrees and extended ED polynomials of algebraic varieties.

Let $\sX\subset\PP^N$ be a projective variety. Recall that $\sU\subset S^2V$ denotes the open dense subsets of nondegenerate symmetric bilinear forms $Q$ on $V$, see Remark \ref{rmk: nondegenerate bilinear forms}. Motivated by Conjecture \ref{conj: main}, it is natural to consider the integer-valued {\em ED degree map}
\begin{equation}\label{eq: ED degree map}
 \Phi_\sX\colon \sU\to\Z_{\ge 0}\,,\quad \Phi_\sX(Q)\coloneqq\EDeg_Q(\sX)\,.
\end{equation}
When $\sX$ is a Segre-Veronese variety $\sV_{\bd,\bn}$, Conjecture \ref{conj: main} is equivalent to
\[
Q_F \in \mathrm{argmin}_{Q\in S^2_+(S^\bd V^\mR)}\Phi_{\sV_{\bd,\bn}}(Q).
\]
\begin{remark}\label{rmk: properties ED degree map}
More in general, for a real projective variety $\sX\subset\PP^N$, it is interesting to study the minimum value $m_\sX$ of the ED degree map $\Phi_\sX$, when its domain is restricted to the convex cone $S_+^2V^\mR$ of inner products in $V^\mR$, and the subset of $S_+^2V^\mR$ where $m_\sX$ is attained. In the other direction, we know that the maximum $M_\sX$ of $\Phi_\sX$ is equal to $\gEDeg(\sX)$. Furthermore, it would be interesting to study the image $\Phi_{\sX}(S^2 V^{\R})$, a subset of the interval $[m_\sX,M_\sX]$. In particular, one might ask if there exists an integer $\alpha\in[m_\sX,M_\sX]\setminus\Phi_{\sX}(S^2 V^{\R})$, namely if $\Phi_{\sX}(S^2 V^{\R})$ contains some gaps. In Section \ref{sec: ED defects} we consider this problem for quadric hypersurfaces and rational normal curves, see respectively Corollaries \ref{cor: minimum ED degree quadric} and \ref{cor: minimum ED degree rational curve}.
\end{remark}

We show an important property of the map $\Phi_\sX$.

\begin{theorem}\label{thm: ED degree map is locally constant}
Let $\sX\subset\PP^N$ be a projective variety. The map $\Phi_\sX$ is lower semicontinuous.
\end{theorem}
\begin{proof}
We prove the theorem using the theory of extended ED polynomials explained in Section \ref{subsec: ED degree ED polynomial}.
Recall that the extended ED polynomial $\mathrm{EDpoly}_{\sX}(u,Q,\varepsilon^2)$ is an element of $\C[u,Q][\varepsilon]$. By Proposition \ref{pro: roots ED polynomial}, its roots $\varepsilon^2$ are precisely of the form $\varepsilon^2=q(u-x)$ for some $u\in V$ and $x\in C\sX$ that is critical for the function $x\mapsto q(u-x)$ on $C\sX$.
Furthermore, for a fixed pair $(\widetilde{u},\widetilde{Q})$ where $\widetilde{Q}\in S^2V$ is nondegenerate and $\widetilde{u}\notin\V(p_{\EDeg_{\widetilde{Q}}(\sX)})$, Proposition \ref{prop: epsilon-degree ED polynomial} implies that
\[
\deg_{\varepsilon^2}\left(\mathrm{EDpoly}_{\sX}(\widetilde{u},\widetilde{Q},\varepsilon^2)\right)=\EDeg_{\widetilde{Q}}(\sX)\,.
\]
Consider the set
\begin{equation}\label{eq: identity epsilon degree and ED degree}
 \Psi_\sX \coloneqq \{(u,Q)\in V\times S^2V \mid\text{either $Q$ is degenerate or $u\in \V(p_{\EDeg_{Q}(\sX)})$}\}\,,
\end{equation}
that can be naturally regarded as a variety in $V\times S^2V$.
Given $(u,Q)\notin \Psi_\sX$, we also have that $(u,Q')\notin \Psi_\sX$ for all $Q'\in S^2V$ that are close enough to $Q$.
Therefore, the $\varepsilon^2$-degree of $\mathrm{EDpoly}_{\sX}(u,Q',\varepsilon^2)$ cannot become smaller than the $\varepsilon^2$-degree of $\mathrm{EDpoly}_{\sX}(u,Q,\varepsilon^2)$, as all coefficients $p_i$ of the extended ED polynomial are polynomials in $u$ and $Q$.
It follows from the identity \eqref{eq: identity epsilon degree and ED degree} that $\EDeg_{Q'}(\sX)\ge\EDeg_{Q}(\sX)$ if $Q'\in S^2V$ is sufficiently close to $Q$. This shows that the ED degree map $\Phi_\sX$ is lower-semicontinuous.
\end{proof}

An immediate consequence of Theorem \ref{thm: ED degree map is locally constant} is the following ``local'' version of Conjecture \ref{conj: main}.

\begin{theorem}\label{thm: local minimality Frobenius ED degree}
Let $Q_F\in S^2_+(S^{\bd} V^{\mR})$ be a Frobenius inner product and consider an inner product $Q\in S^2_+(S^{\bd} V^{\mR})$ in a sufficiently small neighborhood of $Q_F$. Then
\[
\EDeg_Q(\sV_{\bd,\bn}) \ge \EDeg_{Q_F}(\sV_{\bd,\bn})\,.
\]
In other words, any Frobenius inner product $Q_F$ is a local minimum of the ED degree map \eqref{eq: ED degree map}.
\end{theorem}

\begin{remark}\label{rem: approach}
In light of the above discussion, a possible approach for proving our Conjecture \ref{conj: main} is as follows.
One can stratify the set of all $Q\in S^2_+(S^{\bd } V^{\mR})$ according to the associated position of the respective isotropic quadric $\sQ\subset \PP(S^{\bd} V)$ with respect to $\sV_{\bd,\bn}$.
We expect that the ED degree map $\Phi_{\sV_{\bd,\bn }}(Q)$ is constant along each stratum.
Because both the Frobenius inner product $Q_F$ and $\sV_{\bd,\bn}$ are invariant under the standard action of the product $\prod_{i\in [k]} O(V_i, Q_i)$ of orthogonal groups on $S^{\bd } V$, the intersection of the isotropic quadric $\sQ_F$ with $\sV_{\bd,\bn}$ is highly non-transversal.
This suggests that $Q_F$ lies in a low-dimensional stratum of the above stratification.
If one shows that $Q_F$ belongs to the closure (with respect to the Euclidean topology on $S^2_+(S^{\bd } V^{\mR})$) of any stratum, Conjecture \ref{conj: main} would follow from its ``local'' version, that is, Theorem \ref{thm: local minimality Frobenius ED degree}.
\end{remark}

\section{Minimum ED degrees for varieties of rank-one matrices}\label{sec: minimum ED degree Segre-Veronese}

This section aims to prove Conjecture \ref{conj: main} in the case of matrices (bilinear forms) and symmetric matrices using Morse theory. The next result implies Conjecture \ref{conj: main} for $k=2$ and $\bd=(1,1)$.

\begin{theorem}\label{thm: lower bound Q-distance function}
Consider the cone $C\Sigma_\bn\subseteq V=V_1\otimes V_2$ over the Segre variety $\Sigma_\bn\subset\PP(V)$ of rank-1 matrices.
Given any inner product $Q\in S^2V^\mR$ and the associated quadratic form $q\colon V^\mR\to\R$, the distance function
\begin{equation}\label{eq: distance function matrices}
\dist^Q_{C\Sigma_\bn^\mR,u}\colon C\Sigma_\bn^\mR \to\R\,,\quad x\otimes y \mapsto q(u-x\otimes y),
\end{equation}
from a generic matrix $u \in V^\mR$ to the real cone $C\Sigma_\bn^\mR$ has at least $\min(n_1,n_2)+1$ real critical points.
In particular, if $Q_F$ is a Frobenius inner product on $V^\mR$, we have
\begin{equation}\label{eq: lower bound ED degree matrices}
\EDeg_Q(\Sigma_\bn)\ge \EDeg_{Q_F}(\Sigma_\bn)=\min(n_1,n_2)+1\,.
\end{equation}
\end{theorem}


\begin{proof}
Without loss of generality, we assume $n_1\le n_2$. Suppose that the inner product $Q_F$ is induced by the inner products $Q_1$, $Q_2$ on the factors $V_1^\mR$, $V_2^\mR$, respectively. The associated quadratic forms are denoted with $q_1$ and $q_2$.
For $i\in\{1,2\}$, let $\bS^{Q_i}=\{x\in V_i^\mR\mid q_i(x)=1\}$ denote the unit sphere in the Euclidean space $(V_i^\mR, Q_i)$ and let $\bS^Q=\{u\in V_1^\mR\otimes V_2^\mR\mid q(u)=1\}$ be the unit sphere in the inner product space $(V^\mR,Q)$.
Recall that the elements of $C\Sigma_\bn^\mR$ are rank-one matrices $x\otimes y$ for some nonzero $x\in V_1^\mR$ and $y\in V_2^\mR$.
For any rank-one matrix $x\otimes y\in \bS^Q$, we can always assume that $x\in \bS^{Q_1}$: if $q_1(x)\neq 1$, then the assumption $q(x\otimes y)=1$ allows us to say that $q_1(x)\neq 0$, hence we can replace $x$ with $\tilde{x}=x/\sqrt{q_1(x)}$ and $y$ with $\tilde{y}=\sqrt{q_1(x)}\,y$, so that $\tilde{x}\otimes\tilde{y}=x\otimes y$ and $\tilde{x}\in\bS^{Q_1}$. Therefore, the $q_2$-length of $y\in V_2^\mR$ is determined uniquely. Thus, given $x\in\bS^{Q_1}$, we define 
\[
\bS_x^{Q_2}\coloneqq\left\{y\in V_2^\mR\mid x\otimes y\in\bS^Q \right\}.
\]
Note that $\bS_x^{Q_2}\subset V_2^\mR$ is a quadratic hypersurface homeomorphic to the sphere $\bS^{Q_2}$.
Consider now the sphere bundle
\begin{equation}\label{eq: sphere bundle}
E_Q\coloneqq\left\{(x,y)\in V_1^\mR\times V_2^\mR\mid\text{$x\in \bS^{Q_1}$ and $y\in \bS_x^{Q_2}$}\right\}\longrightarrow\bS^{Q_1}
\end{equation}
whose fiber over $x\in \bS^{Q_1}$ is $\bS_x^{Q_2}$.
We now argue that $E_Q\to\bS^{Q_1}$ is trivial, that is, there exists a homeomorphism between $E_Q$ and the product of spheres $\bS^{Q_1}\times \bS^{Q_2}$.
In order to construct it, we fix the basis $\sB=\{e_i\otimes e_j\mid\text{$i\in\{0,\ldots, n_1\}$ and $j\in\{0,\ldots,n_2\}$}\}$ of $V^\mR$, whose elements $e_i\otimes e_j$ we order lexicographically.
We call $\mathrm{vec}$ the map that to a matrix $u=\sum_{i=0}^{n_1}\sum_{j=0}^{n_2} u_{ij}e_i\otimes e_j\in V^\mR$ associates the column-vector $\mathrm{vec}(u)=(u_{00},\ldots, u_{0n_2},u_{10},\ldots,u_{1n_2},\ldots,u_{n_10},\ldots,u_{n_1n_2})^\mT$ of its coordinates in the basis $\sB$.
Let $M_Q$ be the positive definite symmetric matrix of $Q$ in the basis $\sB$. We can write $M_Q=N_Q^\mT N_Q$, where the rows of $N_Q$ are vectorizations of some $(n_1+1)\times (n_2+1)$-matrices $N_{ij}$ for all $i\in\{0,\ldots, n_1\}$ and $j\in\{0,\ldots, n_2\}$.
So, the $Q$-length of $x\otimes y$ satisfies
\begin{align}\label{eq: q(x tensor y)}
\begin{split}
q(x\otimes y) &= \mathrm{vec}(x\otimes y)^\mT M_Q\, \mathrm{vec}(x\otimes y) = \mathrm{vec}(x\otimes y)^\mT N_Q^\mT N_Q\, \mathrm{vec}(x\otimes y)\\
&= \sum_{i=0}^{n_1}\sum_{j=0}^{n_2} (x^\mT N_{ij}\,y)^2 = \sum_{i=0}^{n_1}\sum_{j=0}^{n_2} y^\mT N_{ij}^\mT\,x\,x^T N_{ij}\,y = y^\mT\left(\sum_{i=0}^{n_1}\sum_{j=0}^{n_2} N_{ij}^\mT xx^\mT N_{ij}\right) y\,.
\end{split}
\end{align}
Let $P(x)\coloneqq\sum_{i=0}^{n_1}\sum_{j=0}^{n_2} N_{ij}^\mT xx^\mT N_{ij}$ be the symmetric $(n_2+1)\times (n_2+1)$ matrix appearing in brackets after the last equality in \eqref{eq: q(x tensor y)}. Its entries depend quadratically on $x\in V_1^\mR$. Furthermore, since $q(x\otimes y)=y^\mT P(x)\,y$, the matrix $P(x)$ is positive definite for any $x\in\bS^{Q_1}$.
Let $P(x)=L(x)^\mT L(x)$ be the Cholesky decomposition of $P(x)$; the Cholesky factor $L(x)$ is a unique upper triangular matrix with positive diagonal entries. Furthermore, the dependence of $L(x)$ on $x\in\bS^{Q_1}$ is continuous by \cite[Lemma $12.1.6$]{Schatzman}. All in all, we define a continuous bundle map
\begin{equation}\label{eq: def bundle map phi}
\varphi\colon E_Q\to\bS^{Q_1}\times\bS^{Q_2}\,,\quad(x,y)\mapsto(x,L(x)y)\,.
\end{equation}
The matrix $L(x)$ establishes a linear bijection between fibers $\bS_x^{Q_2}$ and $\bS^{Q_2}$ and $\varphi$ is a homeomorphism that trivializes the bundle \eqref{eq: sphere bundle}.
By K\"unneth formula \cite{Kuenneth}, the Betti numbers of $E_Q$ are
\begin{equation}\label{eq: Betti E_Q}
b_k(E_Q)=
\begin{cases}
1 & \text{for $k\in\{0,n_1+n_2\}$}\\
1 & \text{for $k\in\{n_1,n_2\}$ if $n_1\neq n_2$}\\
2 & \textrm{for $k=n_1=n_2$ if $n_1=n_2$}\\
0 & \text{otherwise.}
\end{cases}
\end{equation}
Recall that any real rank-one tensor on $C\Sigma_\bn^\mR$ can be written as $\lambda\,x\otimes y$, where $\lambda\in\R$ and $(x,y)\in E_Q$. 
By definition, given $u\in V^\mR$, the point $\lambda\,x\otimes y\neq 0$ is an ED critical point of the distance function $\dist^Q_{C\Sigma_\bn^\mR,u}$ if $u-\lambda\,x\otimes y\in N_{\lambda\,x\otimes y}C\Sigma_\bn^\mR$. Since $x\otimes y$ obviously belongs to $T_{\lambda\,x\otimes y}C\Sigma_\bn^\mR$, one must have that 
\begin{align*}
Q(u,x\otimes y)-\lambda\,Q(x\otimes y,x\otimes y)=Q(u,x\otimes y)-\lambda\,q(x\otimes y)=Q(u,x\otimes y)-\lambda=0\,,
\end{align*}
hence the ED critical points relative to $u$ are of the form $Q(u,x\otimes y)\,x\otimes y$.
We now look at other tangent directions to $C\Sigma_\bn^\mR$ at $\lambda\,x\otimes y$.
Observe that $\lambda=Q(u,x\otimes y)\neq 0$.
The total space $E_Q$ of the bundle \eqref{eq: sphere bundle} locally parameterizes the set $\lambda\,\bS^Q\cap C\Sigma_\bn^\mR$ of rank-one tensors of $q$-norm $\lambda$ via the map $\psi_\lambda\colon E_Q\to\lambda\,\bS^Q\cap C\Sigma_\bn^\mR$ defined by $\psi_\lambda(\tilde{x},\tilde{y})\coloneqq\lambda\,\tilde{x}\otimes\tilde{y}$. 
Thus, if $\gamma(t)=(x(t),y(t))$ is a smooth curve in $E_Q$ emanating from $(x,y)=\gamma(0)$, then $\psi_\lambda(\gamma(t))$ is a curve in $\lambda\,\bS^Q\cap C\Sigma_\bn^\mR$ emanating from $\lambda\,x\otimes y=\psi_\lambda(\gamma(0))$. 
In particular, the differential of $\psi_\lambda$ sends the tangent vector $\gamma'(0)=(x'(0),y'(0))\in T_{\gamma(0)}E_Q$ to the tangent vector 
\begin{equation}
\left.\frac{d}{dt}\right|_{t=0}\psi_\lambda(\gamma(t))=\lambda(x'(0)\otimes y+x\otimes y'(0))\in T_{\psi_\lambda(\gamma(0))}(\lambda\,\bS^Q\cap C\Sigma_\bn^\mR)\,.
\end{equation}
Note that $x\otimes y\in N_{\lambda x\otimes y} (\lambda \bS^Q\cap C\Sigma_\bn^\mR)$, since 
the derivative of the identity $Q(\lambda x(t)\otimes y(t), x(t)\otimes y(t))=\lambda$ at $t=0$ reads as $2 Q(\lambda( x'(0)\otimes y(0)+x(0)\otimes y'(0)),x\otimes y)=0$. This implies that $u-\lambda x\otimes y=u-Q(u,x\otimes y) x\otimes y$ belongs to $N_{\lambda\,x\otimes y}(\lambda\bS^Q\cap C\Sigma_\bn^\mR)$ if and only if $u$ does. Hence, the rank-one matrix $Q(u,x\otimes y)\,x\otimes y$ is a (real) ED degree critical point relative to $u$ if and only if $(x,y)\in E_Q$ is a critical point of the bilinear form $F_u(\tilde{x},\tilde{y})=Q(u,\tilde{x}\otimes \tilde{y})$ restricted to $E_Q$. 
Furthermore, $Q(u,x\otimes y)\,x\otimes y\in C\Sigma_\bn^\mR$ is a degenerate critical point of the distance function $\dist^Q_{C\Sigma_\bn^\mR,u}$ if and only if $(x,y)\in E_Q$ is a degenerate critical point of the function $F_u|_{E_Q}$.
Lemma \ref{lem: generic=Morse} then implies that $F_u|_{E_Q}$ is Morse if and only if $u$ is generic.
Under these circumstances, for a given $k\in\{0,1,\ldots,n_1+n_2\}$, let $m_k$ denote the number of critical points of $F_u|_{E_Q}$ of index $k$. Recall that the index of a critical point of a Morse function is the number of negative eigenvalues of the Hessian matrix of the function at this point. Applying Theorem \ref{thm: strong Morse}, we have that
\begin{equation}\label{eq: strong Morse E_Q}
\sum_{k=0}^i (-1)^{i-k} m_k \ge \sum_{k=0}^i (-1)^{i-k} b_k(E_Q)\,,
\end{equation}
where the Betti numbers of $E_Q$ are given by \eqref{eq: Betti E_Q}.
Since $F_u(x,y)=F_u(-x,-y)$, if $(x,y)\in E_Q$ is a critical point of index $k$, then so is $(-x,-y)\in E_Q$. 
Moreover, since $F_u(x,y)=-F_u(-x,y)$, if $(x,y)\in E_Q$ is a critical point of index $k$, then $(-x,y)\in E_Q$ is a critical point of index $n_1+n_2-k$.
In particular, the numbers $m_k$ are all even and $m_k=m_{n_1+n_2-k}$. We have that $m_0\ge b_0=1$, but since $m_0$ is even, then necessarily $m_0\ge 2$.
Next, we have $m_1-m_0\ge b_1-b_0=-1$. Again, by the fact that $m_k$ are even, we have that $m_1-m_0\ge 0$. Summing this inequality with $m_0\ge 2$, we derive that $m_1\ge 2$.
Proceeding in this way, we obtain that for all even $2i<n_1, n_2$ and odd $2j+1<n_1, n_2$,
\begin{align}
\begin{split}
\sum_{k=0}^{2i}(-1)^{2i-k}m_k &\ge \sum_{k=0}^{2i} (-1)^{2i-k} b_k(E_Q) = 1\\
\sum_{k=0}^{2j+1}(-1)^{2j+1-k}m_k &\ge \sum_{k=0}^{2j+1} (-1)^{2j+1-k} b_k(E_Q) = -1\,,
\end{split}
\end{align}
that, together with the evenness of $m_k$, yields 
\begin{equation}\label{eq:localshit}
\sum_{k=0}^{2i}(-1)^{2i-k}m_k\ge 2\,,\quad
\sum_{k=0}^{2j+1}(-1)^{2j+1-k}m_k\ge 0\,.
\end{equation}
Summing these inequalities with $i=j$, we obtain that $m_i\ge 2$ whenever $i<n_1, n_2$.
Now, we have to distinguish the following two cases.

{\it Case 1 ($n_1=n_2=n$)}. 
By \eqref{eq: Betti E_Q} we have that
\begin{align*}
\sum_{k=0}^{n}(-1)^{n-k}m_k \ge \sum_{k=0}^{n}(-1)^{n-k} b_k(E_Q)\ =\ 2+(-1)^n\,.
\end{align*}
Summing this inequality with one of \eqref{eq:localshit} (depending on the parity of $n$), we obtain $m_n\ge 3$. But since $m_k$ is always even, we must have that $m_n\ge 4$.
Since $m_k=m_{2n-k}$, we have that the total number $\sum_{k=0}^{2n}m_k$ of critical points of $F_u|_{E_Q}$ satisfies
\begin{align*}
\sum_{k=0}^{2n}m_k = 2\sum_{k=0}^{n-1} m_k +m_n \ge 4n+4 = 4(n+1)\,.
\end{align*}

{\it Case 2 ($n_1<n_2$)}. Again, \eqref{eq: Betti E_Q} gives the inequality
\begin{align*}
\sum_{k=0}^{n_1} (-1)^{n_1-k} m_k \ge \sum_{k=0}^{n_1}(-1)^{n_1-k} b_k(E_Q) = 1+(-1)^{n_1},
\end{align*}
whose sum with \eqref{eq:localshit} yields that $m_{n_1}\ge 2$.
Since $m_k=m_{n_1+n_2-k}$, we obtain the following bound on the total number of critical points of $F_u|_{E_Q}$:
\begin{align*}
\sum_{k=0}^{n_1+n_2} m_k \ge 2\sum_{k=0}^{n_1} m_k \ge 4(n_1+1)\,.
\end{align*}
We conclude that $F_u|_{E_Q}$ always has at least $4(n_1+1)=4(\min(n_1,n_2)+1)$ critical points. Equivalently, at least $\min(n_1,n_2)+1$ ED critical points relative to $u$ must be real.
\end{proof}

Conjecture \ref{conj: main} also holds for $k=1$ and $\bd=(2)$, that is, in the case of symmetric matrices of any size, as we state next.

\begin{theorem}\label{thm: symmetric_matrices}
Consider the cone $C\sV_{2,n}\subseteq S^2V$ over the Veronese variety $\sV_{2,n}\subset\PP(S^2V)$ of \mbox{rank-1} symmetric matrices.
Given any inner product $Q\in S^2_+(S^2V^\mR)$ and the associated quadratic form $q\colon S^2V^\mR\to\R$, the distance function
\begin{equation}\label{eq: distance function symmetric matrices}
\dist^Q_{C\sV_{2,n}^\mR,u}\colon C\sV_{2,n}^\mR\to\R\,,\quad\pm\,x\otimes x \mapsto q(u\mp x\otimes x)\,,
\end{equation}
from a generic symmetric matrix $u \in S^2V^\mR$ to the real cone $C\sV_{2,n}^\mR$ has at least $n+1$ real critical points.
In particular, if $Q_F$ is a Frobenius inner product on $S^2V^\mR$, we have
\begin{equation}\label{eq: lower bound ED degree symmetric matrices}
\EDeg_Q(\sV_{2,n})\ge \EDeg_{Q_F}(\sV_{2,n})=n+1\,.
\end{equation}
\end{theorem}
\begin{proof}
The proof strategy is similar to the one of Theorem \ref{thm: lower bound Q-distance function}.
The set $\bS^Q\coloneqq\{x\in V^\mR\mid q(x\otimes x)=1\}$ of vectors whose symmetric square has unit $q$-norm is diffeomorphic to an $n$-sphere of an arbitrary norm in $V^\mR$.
A real symmetric matrix of rank one can be represented as $\lambda\,x\otimes x\in C\sV_{2,n}^\mR$ for some $x\in\bS^Q$ and $\lambda\neq 0$.
By definition, if $\lambda\,x\otimes x\neq 0$ is an ED critical point relative to $u\in S^2V^\mR$, then $u-\lambda\,x\otimes x\in N_{\lambda\,x\otimes x}C\sV_{2,n}^\mR$.
Since $x\otimes x\in T_{\lambda\,x\otimes x}C\sV_{2,n}^\mR$, we have 
\begin{equation}
Q(u,x\otimes x)-\lambda\,Q(x\otimes x, x\otimes x) = Q(u,x\otimes x)-\lambda\,q(x\otimes x) = Q(u,x\otimes x)-\lambda = 0\,.
\end{equation}
Furthermore, we have
\begin{align}
\begin{split}
q(u-Q(u,x\otimes x)\,x\otimes x) &= Q(u,u)-2\,Q(u,x\otimes x)Q(u,x\otimes x)+Q(u,x\otimes x)^2\\
&= q(u)-Q(u,x\otimes x)^2\,,
\end{split}
\end{align}
thus $Q(u,x\otimes x)\,x\otimes x$ is an ED critical point relative to $u$ if and only if $x\in\bS^Q$ is a critical point of the function $F_u(x)=Q(u,x\otimes x)$, $x\in \bS^Q$ (note that $\lambda=Q(u,x\otimes x)\neq 0$).
Furthermore, $Q(u,x\otimes x)\,x\otimes x\in C\sV_{2,n}^\mR$ is a degenerate critical point if and only if $x\in \bS^Q$ is a degenerate critical point of $F_u|_{\bS^Q}$. 
Lemma \ref{lem: generic=Morse} then implies that $F_u|_{\bS^Q}$ is Morse if and only if $u$ is generic.
Denoting by $m_k$ the number of critical points of $F_u|_{\bS^Q}$ of index $k\in\{0,1\ldots,n\}$, by Theorem \ref{thm: strong Morse} we get that
\begin{equation}
\sum_{k=0}^i(-1)^{i-k} m_k\ \ge\ \sum_{k=0}^i (-1)^{i-k} b_k(\bS^Q)\,,
\end{equation}
where the Betti numbers of $\bS^Q$ are
\begin{equation}\label{eq: Betti S_Q}
b_k(\bS^Q)=
\begin{cases}
1 & \text{for $k\in\{0,n\}$}\\
0 & \text{otherwise.}
\end{cases}
\end{equation}
Since $F_u(x)=F_u(-x)$, the numbers $m_k$ have to be even.
From this, since $m_0\ge b_0(\bS^Q)=1$, we have that $m_0\ge 2$.
This inequality together with $m_1-m_0\ge b_1-b_0=-1$ and the fact that $m_1$ is even implies that $m_1\ge 2$.
Proceeding so, we obtain for $i\in\{1,\ldots,n-1\}$ that
\begin{equation}
\sum_{k=0}^i (-1)^{i-k} m_k\ \ge\ \sum_{k=0}^i (-1)^{i-k} b_k(\bS^Q)\ =\ (-1)^i\,.
\end{equation}
The previous inequalities and the fact that all $m_k$ are even yield
\begin{align}\label{eq: local_even}
 \sum_{k=0}^i (-1)^{i-k} m_k \ge 2
\end{align}
for even $i$ and
\begin{align}\label{eq: local_odd}
 \sum_{k=0}^i (-1)^{i-k} m_k \ge 0
\end{align}
for odd $i$. Summing these two inequalities for $i-1$ and $i<n$, we obtain that $m_i\ge 2$.
Furthermore, if $n$ is even, we have that $\sum_{k=0}^n (-1)^{n-k} m_k\ge \sum_{k=0}^n (-1)^{n-k} b_k(\bS^Q)=2$, which summed with \eqref{eq: local_odd} for $i=n-1$ yields $m_n\ge 2$.
Analogously, if $n$ is odd, we have that $\sum_{k=0}^n(-1)^{n-k} m_k\ge \sum_{k=0}^n (-1)^{n-k} b_k(\bS^Q)=0$, which summed with \eqref{eq: local_even} for $i=n-1$ gives $m_n\ge 2$.
Summarizing, we have that the Morse function $F_u|_{\bS^Q}$ has at least $m_k\ge 2$ critical points of each Morse index $k\in\{0,1\ldots,n\}$.
Since antipodal critical points $x,-x\in\bS^Q$ of $F_u|_{\bS^Q}$ give rise to the same ED critical point $Q(u,x\otimes x)\,x\otimes x$ relative to $u$, we have that there are at least $n+1$ real ED critical points of $u$. This proves the inequality \eqref{eq: lower bound ED degree symmetric matrices}.
\end{proof}

\section{ED defects of special varieties}\label{sec: ED defects}

In this section we will study the ED defects in more detail for quadric hypersurfaces, rational normal curves, and Veronese surfaces. This investigation will generalize the study of the ED defects of the Segre-Veronese varieties of previous sections to more general projective varieties. It will provide an approach helping us answer the question about finding the inner product that induces the minimum ED degree for a general projective variety, namely Problem~\ref{prob: min ED defect quadric hypersurface}.

\subsection{Quadric hypersurfaces}\label{subsec: ED defects quadrics}
Consider the case $k=2$ and $\bn=\bd=(1,1)$. The Segre variety $\Sigma_2=\Sigma_\bn$ is a smooth quadric surface in $\PP(\C^2\otimes\C^2)=\PP^3_{\mC}$. Since $\Sigma_2=\V(F)$ for some quadratic homogeneous polynomial $F$, we may consider the symmetric matrix $M_F$ associated with the polynomial $F$. We always assume that $\dim(\Sigma_2^\mR)=2$, or rather that the real zero locus of $F$ is a surface in $\PP^3_{\mR}$. Similarly, for a real symmetric bilinear form $Q\in S^2\R^4$, we let $M_Q$ be the symmetric matrix associated with $Q$. As long as $M_Q$ varies in $S^2\C^4$ and is not proportional to $M_F$, the pair $(M_F,M_Q)$ defines a pencil of quadrics in $\PP^3_{\mC}$, that is, a projective line in $\PP(S^2\C^4)$, or rather a point in the Grassmannian $\bG(1,\PP(S^2\C^4))$.
On the one hand, we have already pointed out in the previous sections that the value of $\EDeg_Q(\Sigma_2)$ depends on the intersection between $\Sigma_2$ and the isotropic quadric $\sQ=\V(Q)$. On the other hand, since $M_F$ and $M_Q$ have real entries, to compute the possible values of $\EDeg_Q(\Sigma_2)$, we may pass through the classification of pencils of real symmetric matrices, especially those containing at least one positive definite element.

Let $\K$ be either $\R$ or $\C$ and let $Q_1, Q_2 \in \K[x_0, \ldots, x_n]$ homogeneous of degree two spanning a pencil of quadrics $\{Q_1 - tQ_2 = 0\}$. We can then associate the pencil of Hessian matrices $A_1-tA_2$; linear changes of coordinates correspond to congruences $P(A_1-tA_2)P^\mT$. Hence, to classify pencils of quadrics up to a change of coordinates is to classify pencils of symmetric matrices up to congruence. 

Suppose that $\det(A_2) \neq 0$, and consider the map $\phi\colon \K[t]^{\oplus(n+1)} \to \K[t]^{\oplus(n+1)}$ given by multiplication with $A_1-tA_2$. It follows that $\mathrm{coker}(\phi)$ is a finitely generated torsion $\K[t]$-module and the structure theorem implies that 
\[
\mathrm{coker}(\phi) = \frac{\K[t]}{q_1(t)}\oplus\cdots\oplus\frac{\K[t]}{q_r(t)}\,,
\]
where $q_j(t)$ are primary factors of $\det(A_1-tA_2)$, called the elementary divisors.
For $\K=\C$, we may write $q_j(t) = (t-a_i)^{\sigma_{i,j}}$, and the collection of the integers $\sigma_{i,j}$ form the so-called the \emph{Segre symbol} of the pencil. Two pencils of complex symmetric matrices sharing the same Segre symbol and the same roots $a_i$ are congruent, see \cite[Theorem 1.1]{fevola2021pencils}. 

The real case is more subtle. Besides the elementary divisors, that do not necessarily have real roots, one also needs to consider the inertia. We refer to \cite[Theorem 2]{Thomp} for the classification. Nonetheless, a pencil of real symmetric matrices $\{A_1-tA_2\}$ such that $A_2$ is positive definite is rather special.

\begin{lemma}\label{lem: simult-diag}
Let $\{A_1-tA_2\}$ be a pencil of real symmetric $n\times n$ matrices such that $A_2$ is positive definite. Then there exists $S \in \mathrm{GL}(n, \R)$ such that 
\[
S(A_1-tA_2)S^\mT = \Lambda - t\,I_n\,,
\]
where $\Lambda$ is diagonal with the same eigenvalues as $A_2^{-1}A_1$ and $I_n$ is the identity matrix of order $n$. In particular, only the number $1$ can occur in the Segre symbol. 
\end{lemma}

\begin{remark}
This result is proved, for instance, in \cite[Theorem 7.6.4]{HJ}. We reproduce the argument below. Note that the lemma reduces the classification of pencils of real symmetric matrices containing a positive definite element to the classification of roots and Segre symbols, as in the complex case.
\end{remark}
\begin{proof}[Proof of Lemma \ref{lem: simult-diag}]
A positive definite matrix $A_2$ admits a Cholesky decomposition $A_2 = LL^\mT$. 
Due to the spectral theorem, the real symmetric matrix $L^{-1} A_1 (L^{-1})^\mT$ can be decomposed as $L^{-1} A_1 (L^{-1})^\mT=U \Lambda U^\mT$, where $U \in \mathrm{O}(n,\R)$ is an orthogonal matrix and $\Lambda$ is diagonal. To conclude, define $S \coloneqq U^\mT L^{-1}$.

Note that $A_2^{-1}A_1 = (S^\mT S)(S^{-1} \Lambda (S^\mT)^{-1}) = S^\mT\Lambda (S^\mT)^{-1}$. Hence $A_2^{-1}A_1$ and $\Lambda$ have the same eigenvalues. Furthermore, since $\Lambda - t\,I_n$ is diagonal, the elementary divisors have the form $t-\lambda$, where $\lambda$ is an eigenvalue of $\Lambda$. 
\end{proof}

\begin{example}\label{ex: P1xP1 study defect}
Let us return to the study of $\EDeg_Q(\Sigma_2)$. The radical ideal of $\Sigma_2$ is generated by the homogeneous polynomial $F(x_0,x_1,x_2,x_3)=x_0x_3-x_1x_2$. Hence, the Hessian matrix $M_F$ has signature $(2,2)$. Let $Q$ be another symmetric bilinear form such that its Hessian $M_Q$ is positive definite. Due to Lemma \ref{lem: simult-diag}, $M_F - tM_Q$ is congruent to $\mathrm{diag}(a_0-t, a_1-t, a_2-t, a_3-t)$ for $a_0,a_1, a_2, a_3\in \R \setminus \{0\}$. Imposing that $M_F$ has signature $(2,2)$ we fall into one of the cases below. Given $\sQ=\V(Q)$, we also describe $\Sigma_2\cap\sQ \subset\PP^3_{\mC}$, following \cite[Example 3.1]{fevola2021pencils}, and compute the corresponding ED degree of $\Sigma_2$. 
\begin{enumerate}
 \item $a_0 < a_1 < 0 < a_3 < a_3$: In this case, $\sigma = [1,1,1,1]$ and $\Sigma_2\cap\sQ$ is a smooth elliptic curve. Then $\EDef_Q(\Sigma_2)=0$ and $\EDeg_Q(\Sigma_2)=\gEDeg(\Sigma_2)=6$.
 \item $a_0= a_1 < 0 < a_2<a_3$: In this case, $\sigma = [(1,1),1,1]$ and $\Sigma_2\cap\sQ$ is the union of two noncoplanar conics that meet in two points $P_1,P_2$. By Corollary \ref{cor: ED defect isolated singularities} we have $\EDef_Q(\Sigma_2)=\mu_{P_1}+\mu_{P_2}=2$, namely $\EDeg_Q(\Sigma_2)=6-2=4$.
 \item $a_0 = a_1 < 0 < a_2 = a_3$: In this case $\sigma = [(1,1),(1,1)]$ and $\Sigma_2\cap\sQ$ is the union of four lines, and each line meets two other lines, so $\sZ=\mathrm{Sing}(\Sigma_2\cap\sQ)$ consists of four simple points $P_1,P_2,P_3,P_4$. Note that this is the case when $Q=Q_F$ is the Frobenius inner product, or any product metric, namely an element of $S^2\C^2\otimes S^2\C^2\subset S^2(\C^2\otimes\C^2)$. By Corollary \ref{cor: ED defect isolated singularities} we have $\EDef_{Q_F}(\Sigma_2)=\sum_{i=1}^4\mu_{P_i}=1+1+1+1=4$, namely $\EDeg_{Q_F}(\Sigma_2)=6-4=2$.
\end{enumerate}
Another Segre symbol also involves only the number $1$, namely $[(1,1,1),1]$. However, the associated pencil does not contain any matrix of signature $(2,2)$. Also note that the quadric $\sQ_F$ coming from the Frobenius inner product has the maximal ED defect, as proved in Theorem \ref{thm: lower bound Q-distance function}. Equivalently, we have that
\[
\min_{Q\in S_+^2(\R^2\otimes\R^2)}\Phi_{\Sigma_2}=2\,,
\]
where the ED degree map $\Phi_{\Sigma_2}$ is defined in \eqref{eq: ED degree map}.
Nonetheless, we can get greater ED defects if we allow more general pencils. For instance, for the Segre symbol $[(2,2)]$, the intersection $\Sigma_2\cap\sQ$ is the union of a double line and two other lines that meet the double line at the points $P_1,P_2$. This case is studied in detail in \cite[Example 4.6]{maxim2020defect}, and we have that $\EDef_Q(\Sigma_2)=5$ or, equivalently, $\EDeg_Q(\Sigma_2)=6-5=1$.\hfill$\diamondsuit$
\end{example}

This example tells us that, after fixing a Segre-Veronese variety $\sV_{\bd,\bn}\subset\PP(S^\bd V)$, there may exist smooth complex quadric hypersurfaces $\sQ\subset\PP(S^\bd V)$ such that $\EDef_Q(\sV_{\bd,\bn})$ is greater than $\EDef_{Q_F}(\sV_{\bd,\bn})$. Our Conjecture \ref{conj: main} subsumes that if we restrict to real $\sQ\subset\PP(S^\bd V)$ whose associated matrix $M_Q$ is positive definite, then always $\EDef_Q(\sV_{\bd,\bn})\leq \EDef_{Q_F}(\sV_{\bd,\bn})$.

\begin{example}\label{ex: ED defects Veronese conic}
An analogous case study is also possible for the plane Veronese conic $\sV_{2,1}\subset\PP^2_{\mR}$ that parametrizes $2\times 2$ real symmetric matrices of rank one. It is the zero locus of the homogeneous polynomial $F(x_0,x_1,x_2)=x_0x_2-x_1^2$ whose Hessian matrix $M_F$ has signature $(2,1)$.
Given another conic $\sQ\subset\PP^2$ corresponding to a positive definite matrix $M_Q$, Lemma \ref{lem: simult-diag} implies that $M_F-t M_Q$ is congruent to $\mathrm{diag}(a_0-t,a_1-t,a_2-t)$ for some real $a_0, a_1, a_2\in \R\setminus\{0\}$. Since signatures of congruent matrices coincide, we can distinguish two cases:
\begin{enumerate}
 \item $a_0<a_1<0<a_2$. In this case $\sigma=[1,1,1]$ and the intersection $\sV_{2,1}\cap \sQ$ consists of $4$ reduced points. Then $\EDef_Q(\sV_{2,1})=0$ and $\EDeg_Q(\sV_{2,1})=\gEDeg(\sV_{2,1})=4$. 
 \item $a_0=a_1<0<a_2$. Here $\sigma=[(1,1),1]$ and $\sV_{2,1}\cap \sQ$ consists of $2$ double points $P_1,P_2$. We have that $\EDef_Q(\sV_{2,1})=\mu_{P_1}+\mu_{P_2}=2$ and hence $\EDeg_Q(\sV_{2,1})=4-2=2$. This case corresponds to the Frobenius inner product.
\end{enumerate}
The previous argument reproves 
\[
\min_{Q\in S_+^2(S^2\R^2)}\Phi_{\sV_{2,1}}=2\,,
\]
that is anticipated from Theorem \ref{thm: symmetric_matrices}.\hfill$\diamondsuit$
\end{example}

An instance of the general problem formulated in Remark \ref{rmk: properties ED degree map} is the following.

\begin{problem}\label{prob: min ED defect quadric hypersurface}
Let $\sX$ be a fixed smooth real quadric hypersurface in $\PP(V^\mR)$.
Determine
\[
\min_{Q\in S_+^2V^\mR}\Phi_\sX(Q)\,,
\]
or equivalently the maximum possible ED defect of $\sX$ among all inner products $Q\in S_+^2V^\mR$.
\end{problem}

A result \cite[Eq. (7.1)]{DHOST} related to Problem \ref{prob: min ED defect quadric hypersurface} implies that the generic ED degree of a smooth quadric hypersurface $\sX\subset\PP^N$ is $2N$.

\begin{proposition}\label{prop: ED degree X quadric}
Let $\sX=\V(F)$ and $\sQ=\V(Q)$ be smooth quadric hypersurfaces in $\PP^N$ such that the associated symmetric matrices $M_F$ and $M_Q$ are real. Furthermore, assume that $M_Q$ is positive definite. Then $\EDeg_Q(\sX)= 2(r-1)$, where $r$ is the number of distinct eigenvalues of $M_Q^{-1}M_F$.
\end{proposition}
\begin{proof}
Due to Lemma \ref{lem: simult-diag} the pencil of matrices $M_F + tM_Q$ is congruent to $\Lambda - t\,I_{N+1}$ where $\Lambda$ has the same eigenvalues as $M_Q^{-1}M_F$, say, $a_1, \ldots, a_r$ are the distinct ones. Let $\{e_1, \ldots, e_{N+1}\}$ denote the chosen basis of $V$ and define $\Gamma_i\coloneqq\{\, j \in[N+1] \mid \Lambda e_j = a_i e_j\}$. Then the quadrics $F$ and $Q$ in these coordinates are, respectively,
\[
F = \sum_{i=1}^r a_i\sum_{j\in \Gamma_i}x_j^2 \quad \text{ and } \quad Q = \sum_{i=1}^{N+1} x_i^2.
\]
It follows from the Jacobian ideal computation that $\sX\cap\sQ$ is singular along $\sS_i = \V(\{x_j\mid j\not\in\Gamma_i\})\cap\sQ\subset\PP^N$, which is not empty only if $\ell_i\coloneqq|\Gamma_i|\ge 2$. Furthermore, note that $\sS_i$ are smooth and pairwise disjoint. Indeed, for $i\neq j$ we have $\Gamma_i \cap \Gamma_j = \emptyset$, hence the ideal of $\sS_i\cap\sS_j$ contains $x_1, \ldots, x_{N+1}$, so $\sS_i\cap\sS_j$ is empty. Moreover, up to reordering the basis, we may assume that $\Gamma_i = \{1, \ldots, l_i\}$, hence
\[
\sS_i = \V(x_1^2+\ldots + x_{l_i}^2 , x_{l_i+1} , \ldots, x_{N+1})\,,
\]
and it is clear that $\sS_i$ is smooth.

The previous considerations tell us that $\EDef_Q(\sX)$ is the sum of $r$ contributions $c_i$, each coming from every variety $\sS_i$. We claim that $c_i = 2(\ell_i -1)$ for all $i\in[r]$. Indeed, if $\ell_i \ge 3$ then $\sS_i$ is a smooth quadric hypersurface in $\V(\{x_j\mid j\not\in\Gamma_i\})\cong\PP^{\ell_i-1}$. Then, due to Corollary \ref{cor: ED defect equisingular}, $c_i = \gEDeg(\sS_i) = 2(\ell_i-1)$, where the last equality comes from \cite[Eq. (7.1)]{DHOST}. If $\ell_i = 2$, then $\sS_i$ gives two simple isolated singularities, hence $c_i = 2= 2(2-1)$. By abuse of notation, if $\ell_i = 1$, $\sS_i = \emptyset$ and contributes to the ED defect with $c_i=2(\ell_i-1) = 0$. Summing up,
\[
\EDef_Q(\sX) = \sum_{i = 1}^r 2(\ell_i-1) = \sum_{i = 1}^r 2\,\ell_i-2\,r = 2(N+1-r)\,. 
\]
Therefore $\EDeg_Q(\sX)=\gEDeg(\sX)-\EDef_Q(\sX)=2N-2(N+1-r)=2(r-1)$.\qedhere
\end{proof}

\begin{corollary}\label{cor: minimum ED degree quadric}
Consider the hypotheses of Proposition \ref{prop: ED degree X quadric}, and assume that $\sX^\mR\neq\emptyset$. Then the image of the ED degree map $\Phi_\sX$, when restricted to $S_+^2V^\mR$, is
\[
\Phi_\sX(S_+^2V^\mR) = \{2,4,\ldots,2N-2,2N\}\,.
\]
In particular, $\min_{Q\in S_+^2V^\mR}\Phi_\sX(Q) = 2$.
\end{corollary}
The minimum $2$ of $\Phi_\sX$ is attained precisely when $M_Q^{-1}M_F$ has two distinct eigenvalues, it includes what Aluffi and Harris call a ``sphere'', see \cite[\S9.2]{aluffi-harris}.
On the other hand, if $\sX^\mR=\emptyset$, namely $\sX=\V(\widetilde{Q})$ for some $\widetilde{Q}\in S_+^2V^\mR$, then $\EDeg_{\widetilde{Q}}(\sX)=0$ by Proposition \ref{prop: ED degree X quadric}.

\subsection{Rational normal curves}\label{subsec: ED defects rational normal curves}

Coming back to Segre-Veronese varieties, let us consider now the case $k=1$, $\bd=(d)$, and $\bn=(1)$, namely the case of a rational normal curve $\sV_{d,1}\subset\PP(S^dV)\cong\PP^d_{\mC}$. By Theorem \ref{thm: gen ED degree Segre-Veronese} we get that
\begin{equation}\label{eq: gEDdegree rational curve}
\gEDeg(\sV_{d,1})=3d-2\,.
\end{equation}
Given any inner product $Q$ on $S^dV^\mR\cong\R^{d+1}$, if $\sQ=\V(Q)\subset\PP^d_{\mC}$ does not contain $\sV_{d,1}$, then the intersection $\sV_{d,1}\cap\sQ$ consists of $2d$ points counted with multiplicity. If $\sQ$ is generic, then $\sV_{d,1}\cap\sQ$ consists of $2d$ distinct points, namely $\EDef_{Q}(\sV_{d,1})=0$. Assume that $Q=Q_F$, where $Q_F$ is a Frobenius inner product. Without loss of generality, we can assume that $Q_F$ is obtained as the tensor power of the standard Euclidean inner product on $V^\mR\cong\R^2$.
Hence, if we consider the parametrization 
\begin{equation}\label{eq: parametrization rational normal curve}
\nu_d([x_0:x_1])=[x_0^d:x_0^{d-1}x_1:\cdots:x_0x_1^{d-1}:x_1^d]
\end{equation}
of $\sV_{d,1}$, then the isotropic quadric $\sQ_F$ has the equation
\begin{equation}\label{eq: equation Q_F case rational normal curve}
\sum_{i=0}^d\binom{d}{i}z_i^2=0\,,
\end{equation}
where $z_0,\ldots,z_d$ are homogeneous coordinates of $\PP^d$.
To compute the intersection $\sV_{d,1}\cap\sQ_F$, we substitute the relations $z_i=x_0^{d-i}x_1^i$ in \eqref{eq: equation Q_F case rational normal curve}. This produces the identity
\[
0=(x_0^2+x_1^2)^d=(x_0+\sqrt{-1}x_1)^d(x_0-\sqrt{-1}x_1)^d\,.
\]
Hence $\sV_{d,1}\cap\sQ_F$ consists of the two points $\nu_d([1:\sqrt{-1}])$ and $\nu_d([1:-\sqrt{-1}])$ with multiplicity $d$. By Corollary \ref{cor: ED defect isolated singularities} we conclude that $\EDef_{Q_F}(\sV_{d,1})=2(d-1)$, that is $\EDeg_{Q_F}(\sV_{d,1})=d$. One may verify that this value is equal to the one obtained, for example, in Theorem \ref{thm: FO formula}.
In general, we have the following result. 

\begin{proposition}\label{prop: ED defect rational normal curve}
Consider the real rational normal curve $\sV_{d,1}^\mR$ parametrized by the map \eqref{eq: parametrization rational normal curve}, and its associated complex curve $\sV_{d,1}\subset\PP^d_{\mC}$. Let $Q$ be a nondegenerate symmetric bilinear form on $S^dV^\mR$. If its associated real symmetric matrix $M_Q$ is positive definite, then $\EDef_{Q}(\sV_{d,1})$ is an even integer. Furthermore, if $\EDef_{Q}(\sV_{d,1})\ge 2d-1$, then the matrix $M_Q$ cannot be chosen real positive definite.
\end{proposition}
\begin{proof}
Consider the variables $x_0,x_1$ and the vector $v=(x_0^d,x_0^{d-1}x_1,\ldots,x_0x_1^{d-1},x_1^d)$. Suppose first that $\sV_{\bd,\bn}$ is not contained in the quadric $\sQ=\V(Q)$, hence the intersection $\sV_{d,1}\cap\sQ$ consists of $2d$ complex points, counted with multiplicity. These points in $\PP(S^d V)$ correspond in $\PP(V)\cong\PP^1_{\mC}$ to the complex roots of the binary form $b(x_0,x_1)=vM_Qv^\mT$ of degree $2d$. Since $b(x_0,x_1)$ has real coefficients, it has either real roots or complex roots coming in conjugate pairs.

Suppose now that $M_Q$ is real positive definite. Then the equation $b(x_0,x_1) = 0$ cannot have nonzero real solutions $(x_0,x_1)$, therefore it has $t\ge 1$ distinct pairs of nonreal conjugate roots with multiplicities $m_1,\ldots,m_t$ such that $m_1+\cdots+m_t=d$. Hence, by Corollary \ref{cor: ED defect isolated singularities}, we have that
\begin{equation}\label{eq: even defect}
\EDef_{Q}(\sV_{d,1})=\sum_{i=1}^t 2(m_i-1)=2(d-t)\,.
\end{equation}
In particular, the first part of the statement follows. Furthermore, the maximum possible defect is attained for $t=1$ and is equal to $2d-2$, and this value is strictly smaller than the maximum ED defect of $\sV_{d,1}$ when is not contained in $\sQ$, that is equal to $2d-1$.
Instead, when $\sV_{d,1}\subset\sQ$, then $\EDeg_Q(\sV_{d,1})=0$ or $\EDef_Q(\sV_{d,1})=3d-2$.
In this case, we also have $vM_Qv^T=0$ for all vectors $v=(x_0^d,x_0^{d-1}x_1,\ldots,x_0x_1^{d-1},x_1^d)$. Choosing $(x_0, x_1) = (1,0)$, we get a nonzero real solution of $b(x_0,x_1)=0$, hence $M_Q$ cannot be positive definite.
\end{proof}

\begin{remark}
If $M_Q$ is positive definite, then $Q$ has no real points. The converse is true up to the choice of $\pm M_Q$. Sylvester's law of inertia reduces the problem to looking at the possible signatures; if $M_Q$ has both positive and negative eigenvectors, we can cook up a real point in $Q$.
\end{remark}

From the proof of Proposition \ref{prop: ED defect rational normal curve} and \eqref{eq: gEDdegree rational curve}, we also derive the following corollary.

\begin{corollary}\label{cor: minimum ED degree rational curve}
Consider the hypotheses of Proposition \ref{prop: ED defect rational normal curve}. Then the image of the ED degree map $\Phi_{\sV_{d,1}}$, when restricted to $S_+^2(S^dV^\mR)$, is
\[
\Phi_{\sV_{d,1}}(S_+^2(S^dV^\mR)) = \{d,d+2,\ldots,3d-2\}\,.
\]
In particular, $\min_{Q\in S_+^2(S^dV^\mR)}\Phi_{\sV_{d,1}}(Q) = d$.
\end{corollary}

\begin{remark}
For any $d\ge 1$, we can write an explicit example of a quadric $\sQ_d=\V(Q_d)$ such that $\EDef_{Q_d}(\sV_{d,1})=2d-1$. Indeed, consider the quadric $Q_d$ of equation
\begin{equation}\label{eq: equation Q_d}
z_0^2-\left(\sum_{i=1}^{\lfloor\frac{d-1}{2}\rfloor}z_iz_{d+1-i}\right)+\left\lfloor\frac{d-1}{2}\right\rfloor z_{\lfloor\frac{d+1}{2}\rfloor}z_{\lceil\frac{d+1}{2}\rceil}=0\,.
\end{equation}
The symmetric matrix $M_d$ associated with $Q_d$ has two diagonal blocks: the $1\times 1$ diagonal block with entry $1$, and the $d\times d$ diagonal block that is an anti-diagonal matrix. The vectors $e_0=(1,0,\ldots,0)$ and $e_1=(0,1,\ldots,0)$ are eigenvectors of $M_d$ with eigenvalues $1$ and $-\frac{1}{2}$, respectively. This means that $M_d$ is not positive definite. 
Plugging in the relations $z_i=x_0^{d-i}x_1^i$ in \eqref{eq: equation Q_d}, we get
\begin{align}
\begin{split}
0 &= x_0^{2d}-\left(\sum_{i=1}^{\lfloor\frac{d-1}{2}\rfloor}x_0^{d-i}x_1^ix_0^{i-1}x_1^{d+1-i}\right)+\left\lfloor\frac{d-1}{2}\right\rfloor x_0^{d-\lfloor\frac{d+1}{2}\rfloor}x_1^{\lfloor\frac{d+1}{2}\rfloor}x_0^{d-\lceil\frac{d+1}{2}\rceil}x_1^{\lceil\frac{d+1}{2}\rceil}\\
&= x_0^{2d}-\left\lfloor\frac{d-1}{2}\right\rfloor x_0^{d-1}x_1^{d+1}+\left\lfloor\frac{d-1}{2}\right\rfloor x_0^{d-1}x_1^{d+1} = x_0^{2d}\,, 
\end{split}
\end{align}
where we used the identity $\lfloor\frac{d+1}{2}\rfloor+\lceil\frac{d+1}{2}\rceil=d+1$.
Hence, $\sV_{d,1}\cap\sQ_d$ consists of the point $\nu_d([0,1])$ with multiplicity $2d$ and $\EDef_{Q_d}(\sV_{d,1})=2d-1$.
\end{remark}

\subsection{Veronese surfaces}\label{subsec: ED defects Veronese surfaces}

Consider the Veronese surface $\sV_{d,2}\subset\PP(S^d\C^3)\cong\PP^{\binom{d+2}{2}-1}_{\mC}$. By Theorem \ref{thm: gen ED degree Segre-Veronese} (see also \cite[Proposition 7.10]{DHOST} for arbitrary Veronese varieties) we have
\begin{equation}
\gEDeg(\sV_{d,2})=7d^2-9d+3\,.
\end{equation}
Let $Q$ be an inner product on $S^d\R^3$, whose associated complex quadric hypersurface in $\PP(S^d\C^3)$ is $\sQ$.
The intersection $\sV_{d,2}\cap\sQ$ has degree $2d^2$ and can be written as the image of a plane curve $ \sC$ of degree $2d$ under the Veronese embedding $\nu_d$. Here we used the property that, for any curve $\sC\subset\PP^2_{\mC}$ of degree $k$, $\deg(\nu_d(\sC)) = dk$. Hence, using Theorem \ref{thm: ED defect general} and its corollaries, the classification of the possible ED defects $\EDef_Q(\sV_{d,2})$ corresponds to the classification of all plane curves of degree $2d$. 

\begin{example}[Second Veronese Surface]
We focus on the case $d=2$, hence $\sV_{2,2}$ is the second Veronese surface in $\PP^5_{\mC}$.
The intersection $\sV_{2,2}\cap\sQ$ has degree $8$ and can be written as $\nu_2(\sC)$ for some plane curve $\sC$ of degree $4$.
If $Q=Q_F$ is a Frobenius inner product, the curve is special and consists of a double conic $2\sD$. Thus, Corollary \ref{cor: ED defect equisingular} implies
\begin{equation}\label{eq: ED defect V22}
\EDef_{Q_F}(\sV_{2,2}) = \gEDeg(\nu_2(\sD)) = 3\cdot4-2 = 10\,.
\end{equation}
Now assume that $Q$ is such that $\sC=\sV_{2,2}\cap\sQ$ is reduced. The total sum of the Milnor numbers of its singularities is at most $7$, unless $\sC$ is the union of four concurrent lines, with only one singularity with $\mu = 9$, see \cite{shin-bound}.
Due to Corollary \ref{cor: ED defect isolated singularities}, we have $\EDef_{Q_F}(\sV_{2,2}) \leq 9$ in this case.
Next, we deal with the other nonreduced cases. Below, $\sD$ denotes a plane conic, while $\sL$ and $\sL_i$ are lines.
\begin{enumerate}
 \item $\sC = \sD \cup 2\sL$, with $\sD\cap\sL = \{P_1,P_2\}$. Then $\mathrm{Sing}(\sC)$ has three strata: $\sS_0 = \sL\setminus\{P_1,P_2\}$ and $\sS_i = \{P_i\}$, $i=1,2$. The Milnor fiber $F_{\sS_0}$ around any point of $\sS_0$ is homotopic to $\{x^2 = 1\}$, hence $\chi(F_{\sS_0}) = 2$ and $\mu_{\sS_0} = 1$. Thus $\alpha_{\sS_0} = 1$. For $\sS_1$ or $\sS_2$ the Milnor fiber is homotopic to $\{x^2y = 1\}$, hence $\mu_{\sS_1} = \mu_{\sS_2} = -1$. The complex link in this case is just a point (see \cite[p.492]{rod-wang-maxlike} for the definition) thus $\alpha_{\sS_i} = \mu_{\sS_i} - \mu_{\sS_0} = -2$. Therefore,
 \[
 \EDef_{Q}(\sV_{2,2}) = 4\alpha_{\sS_0} - \alpha_{\sS_1} -\alpha_{\sS_2} = 4 -(-2) -(-2) = 8\,.
 \] 
 \item $\sC=\sD\cup 2\sL$, with $\sD\cap\sL = \{2P\}$, a tangent line. Then $\mathrm{Sing}(\sC)$ has two strata: $\sS_0 = \sL\setminus\{P\}$ and $\sS_1 = \{P\}$. As in the previous case $\alpha_{\sS_0} = 1$. For $\sS_1$, we may use the formula in \cite{melle-euler} to compute $\chi(F_{\sS_1}) = - 2\cdot3 + 3 = -3$ hence $\mu_{\sS_1} = -4$ and $\alpha_{\sS_1} = -4 - 1 = -5$. Therefore,
 \[
 \EDef_{Q}(\sV_{2,2}) = 4\alpha_{\sS_0} - \alpha_{\sS_1} = 4 -(-5) = 9\,.
 \] 
 \item $\sC= \sL_1 \cup \sL_2 \cup 2\sL_3$ with $\sL_1\cap \sL_2 \cap \sL_3 = \emptyset$. This is similar to case $(1)$ but with an extra isolated singularity $A_1$. Therefore,
 \[
 \EDef_{Q}(\sV_{2,2}) = 8 + 1 = 9\,.
 \] 
 \item $\sC= \sL_1 \cup \sL_2 \cup 2\sL_3$ with $\sL_1\cap \sL_2 \cap \sL_3 = \{P\}$. We have two strata: $\sS_0 = \sL_3\setminus \{P\}$ and $\sS_1 = \{P\}$. Moreover, $\alpha_{\sS_0} = 1$ and $\chi(F_{\sS_1}) = - 4$. Therefore,
 \[
 \EDef_{Q}(\sV_{2,2}) = 4\alpha_{\sS_0} - \alpha_{\sS_1} = 4 -(-5 -1) = 10\,.
 \]
 \item $\sC= 2\sL_1 \cup 2\sL_2 $ with $\sL_1\cap \sL_2 = \{P\}$. We have three strata: $\sS_0 = \sL_1 \setminus \{P\}$, $\sS_1 = \sL_2 \setminus \{P\}$, and $\sS_2 = \{P\}$. As in previous cases, $\alpha_{\sS_0} = \alpha_{\sS_1} = 1$. For $\sS_2$ we get $\chi(F_{\sS_2}) = 0$ hence $\alpha_{\sS_2} = -1 - \alpha_{\sS_0} - \alpha_{\sS_1} = -3$. Then
 \[
 \EDef_{Q}(\sV_{2,2}) = 4\alpha_{\sS_0} + 4\alpha_{\sS_1} - \alpha_{\sS_2} = 4+ 4 -(-3) = 11\,.
 \]
 \item $\sC= 3\sL_1 \cup \sL_2 $ with $\sL_1\cap \sL_2 = \{P\}$. We have two strata: $\sS_0 = \sL_1\setminus \{P\}$ and $\sS_2 = \{P\}$. Now $\alpha_{\sS_0} = 2$ and $\mu_{\sS_1} = -1$. Then
 \[
 \EDef_{Q}(\sV_{2,2}) = 4\alpha_{\sS_0} - \alpha_1 = 4\cdot 2 - ( -1 - 2) = 11\,.
 \]
 \item $\sC= 4\sL$. Then we may apply Corollary \ref{cor: ED defect equisingular} to get
 \[
 \EDef_{Q}(\sV_{2,2}) = 3\gEDeg(\nu_2(\sL)) = 12\,.
 \]
\end{enumerate}
From Theorem \ref{thm: symmetric_matrices} the maximal ED defect coming from a real positive definite inner product is $10$, as we computed in \eqref{eq: ED defect V22}. Equivalently, the minimum ED degree is $\EDeg_{Q_F}(\sV_{2,2})=3$. Thus, cases (5), (6), and (7) above cannot be realized by a definite bilinear form $Q$. In fact, none of these cases can be realized by positive quadrics. Indeed, note that the Veronese surface $\sX$ is defined over $\R$ and so is $\sC= \sV_{2,2}\cap\sQ$, whenever $\sQ$ is real. If $Q$ is a definite quadric, then it has no real (closed) points, hence neither does $\sC$. However, every curve above has real (closed) points, as it follows from the lemma below.\hfill$\diamondsuit$
\end{example}

\begin{lemma}\label{lem: odd has real points}
Let $\sC\subset\PP^2_{\mC}$ be a projective plane curve defined over $\R$. If either $\sC$ contains an odd degree subcurve (over $\C$) or $\sC$ has an odd number of isolated singular points, then $\sC$ contains a real point, i.e., $C^\mR \neq \emptyset$.
\end{lemma}
To some extent, the idea behind this statement is that ``a real one-variable polynomial of odd degree has at least one real root''.
\begin{proof}
Denote by $\bar{z}$ the complex conjugate of $z\in\C$. Then $z\mapsto\bar{z}$ generates the Galois group $\mathrm{Gal}(\C/\R)$. For any scheme $\sX$ defined over $\R$, $\mathrm{Gal}(\C/\R)$ acts on the irreducible components of the associated complex scheme $\sX^\mC$ and the orbits are defined over $\R$, i.e., if $T$ is an irreducible component of $\sX^\mC$, then there exists another component $\overline{T} \cong T$ such that $T\cup \overline{T}$ is defined over $\R$. In particular, if $T$ is fixed by $\mathrm{Gal}(\C/\R)$, then $T$ is defined over $\R$. For a proof of this fact in more generality, see \cite[\href{https://stacks.math.columbia.edu/tag/04KY}{Lemma 04KY}]{stacks-project}. In particular, if $\sX$ is zero-dimensional of odd length, the action of $\mathrm{Gal}(\C/\R)$ has fixed points, hence $\sX$ has real (closed) points. 

In our case, let $\sC$ be a curve defined over $\R$. Assume, without loss of generality, that $\sC=\sC_{\rm red}$ is reduced. Note that $\mathrm{Sing}(\sC)$ is also defined over $\R$. By our previous discussion, if $\mathrm{Sing}(\sC)$ has odd length, then it has real points, which are contained in $\sC$.
 
Now let $\sC=\sC_1\cup\cdots\cup\sC_k$ be a decomposition of $\sC$ into irreducible components. Suppose that $\deg(\sC_i) \equiv 1 \pmod{2}$ for some $i\in[k]$. Then $\sC_i \cup \overline{\sC_i} \subset \sC$ is defined over $\R$. If $\sC_i = \overline{\sC_i}$, i.e., $\sC_i$ is defined by a real homogeneous polynomial, then take $\sL\subset\PP^2$ a general line, then take $\sX = \sL\cap\sC_i$. If $\sC_i \neq \overline{\sC_i}$ then take $\sX = \sC_i \cap \overline{\sC_i}$. In either case, $\overline{\sX} = \sX$ is defined over $\R$, and it has odd length, due to B\'ezout's Theorem. Thus $\sX$ has real points, hence so does $\sC$.
\end{proof}

\begin{lemma}\label{lem:defVerSurf}
For all $d\ge 2$, consider the Veronese surface $\sV_{d,2}\subset\PP(S^d\C^3)\cong\PP^{\binom{d+2}{2}-1}_{\mC}$.
Let $Q$ be a symmetric bilinear form on $S^d\R^3$ such that $\EDef_{Q}(\sV_{d,2}) \ge \EDef_{Q_F}(\sV_{d,2})$, where $Q_F$ is a Frobenius inner product on $S^d\R^3$. Given $\sQ=\V(Q)$, then the curve $\sQ\cap\sV_{d,2}$ is nonreduced.
\end{lemma}
\begin{proof}
We show the contrapositive. For any quadric hypersurface $\sQ\subset\PP^{\binom{d+2}{2}-1}_{\mC}$ we have that $\sQ\cap\sV_{d,2} = \nu_d(\sC)$, where $\sC\subset\PP^2_{\mC}$ is a plane curve of degree $2d$. If $\sQ = \sQ_F$ is a Frobenius quadric, then $\sC= d\sD$, where $\sD$ is a smooth conic. It follows from Corollary \ref{cor: ED defect equisingular} that
\[
\EDef_{Q_F}(\sV_{d,2}) = (d-1)\gEDeg(\nu_d(\sD)) = (d-1)(6d-2)\,.
\]
Now consider a quadric $\sQ$ such that $\sC$ is reduced. Then Theorem \ref{cor: ED defect isolated singularities} says that the ED defect is the sum of the Milnor numbers of $\sC$, and a bound for this number was proved in \cite{shin-bound}. Hence,
\[
\EDef_{Q_F}(\sV_{d,2}) = \sum_{P \in \sC} \mu_P \leq (2d -1)^2\,.
\]
It is clear that $(2d -1)^2 < (d-1)(6d-2)$ for $d\ge 2$, concluding the proof.
\end{proof}

\begin{corollary}\label{cor:3-rd Veronese}
Conjecture \ref{conj: main} holds for the 3-rd Veronese surface $\sV_{3,2}\subset\PP(S^3\C^3)$, namely
\begin{equation}
 \min_{Q\in S_+^2(S^3\R^3)}\Phi_{\sV_{3,2}}(Q) = \EDeg_{Q_F}(\sV_{3,2})= 7\,.
\end{equation}
\end{corollary}
\begin{proof}
Let $Q$ be a quadric such that $\EDef_{Q}(\sV_{3,2}) \ge \EDef_{Q_F}(\sV_{3,2})$, and let $\sC$ be the preimage of $\sV_{3,2} \cap \sQ$ under the Veronese embedding $\nu_3$. In particular $\deg(\sC)=6$. Let $F$ be a homogeneous polynomial defining $\sC\subset\PP^2_{\mC}$ and let $F = F_1^{\beta_1} \cdots F_k^{\beta_k}$ be the factorization of $F$ into irreducible components. Then we associate to $\sC$ a partition of $6$ by putting $n_i$ times $\deg (F_i)$. From Lemmas \ref{lem: odd has real points} and \ref{lem:defVerSurf} we know that the admissible partitions contain only even numbers and have repetitions. This is enough to conclude that the partition associated with $\sC$ is $\{2,2,2\}$, i.e., $\sC$ is a triple conic. Therefore $\EDef_{Q}(\sV_{3,2}) = \EDef_{Q_F}(\sV_{3,2})$.
\end{proof}

\begin{remark}
The same approach applied to the $4$-th Veronese surface $\sV_{4,2}$ yields the admissible partitions $\{4,4\}$, $\{4,2,2\}$, and $\{2,2,2,2\}$. The last case is associated with a Frobenius inner product. However, a bound for the ED defect in the first two cases still eludes us.
\end{remark}

\section{Nontransversality conditions and leading coefficient of the ED polynomial}\label{sec: ED polynomial}

In Section \ref{subsec: ED degree ED polynomial} we introduced the extended ED polynomial of a projective variety. We used it to prove that the ED degree map $\Phi_\sX$ in \eqref{eq: ED degree map} is lower-semicontinuous. In this section, we investigate more on the polynomial coefficients $p_i(u,Q)$ in the expansion of the extended ED polynomial given in \eqref{eq: write extended ED polynomial explicitly}. The starting point is the following result.

\begin{proposition}\label{prop: degree and leadcoef extended ED polynomial}
Consider a variety $\sX\subset\PP(V)=\PP^N$ and the expansion \eqref{eq: write extended ED polynomial explicitly} of the extended ED polynomial of $\sX$. The $\varepsilon^2$-degree $D$ in $\mathrm{EDpoly}_{\sX}(u,Q,\varepsilon^2)$ equals $\gEDeg(\sX)$. Furthermore, the leading coefficient $p_{\gEDeg(\sX)}(u,Q)=p_{\gEDeg(\sX)}(Q)\in\C[Q]$ depends only on $Q$.
\end{proposition}
\begin{proof}
The first part of the statement is an immediate consequence of Proposition \ref{prop: epsilon-degree ED polynomial}. Regarding the second part, by \cite[Corollary 4.6]{ottaviani2020distance} for any choice of a nondegenerate bilinear form $\widetilde{Q}$, if $\sX$ and the isotropic quadric $\widetilde{\sQ}=\V(\widetilde{Q})$ intersect transversally, then $\EDeg_{\widetilde{Q}}(\sX)=\gEDeg(\sX)$ and each polynomial coefficient $p_i(u,\widetilde{Q})$ of $\varepsilon^{2i}$ in \eqref{eq: write extended ED polynomial explicitly} is a homogeneous polynomial in the coordinates of $u$ of degree $2\gEDeg(\sX)-2i$. In particular $p_{\gEDeg(\sX)}(u,\widetilde{Q})=p_{\gEDeg(\sX)}(\widetilde{Q})\in\C$. This implies that $p_{\gEDeg(\sX)}$ is a polynomial depending only on $Q$.
\end{proof}

The main consequence of Proposition \ref{prop: degree and leadcoef extended ED polynomial} is that, given a variety $\sX\subset\PP(V)$, the zero locus $\V(p_{\gEDeg(\sX)})\subset\PP(S^2V)$ describes those inner products $Q\in S^2V$ such that $\EDeg_Q(\sX)<\gEDeg(\sX)$. Via the following Proposition~\ref{prop: variety Q not transversal X irreducible} and Proposition~\ref{pro: nonzero leadcoeff ED polynomial Q symbolic}, our goal is to describe this variety and compute its degree. A careful study of this variety would help us understand better the locus of inner products that induce smaller ED degrees of $\sX$.

\begin{proposition}\label{prop: variety Q not transversal X irreducible}
Let $\sX\subset\PP^N$ be an irreducible variety of dimension $n$. For any $e\ge 1$, we consider the space $\PP(S^eV^*)$ that parameterizes all hypersurfaces in $\PP(V)$ of degree $e$. Define
\begin{equation}\label{eq: def Z d X}
\sZ_{e,\sX}\coloneqq\overline{\{\sY\in\PP(S^eV^*)\mid\text{$\exists\,x\in \sX_{\mathrm{sm}}\cap \sY_{\mathrm{sm}}$ such that $T_x\sY\supseteq T_x\sX$}\}}\,.
\end{equation}
The variety $\sZ_{e,\sX}$ coincides with the dual variety of the $e$-th Veronese embedding of $\sX_e\coloneqq\nu_e(\sX)\subset\PP(S^eV)$, in particular it is irreducible. Furthermore, $\sZ_{e,\sX}$ is always a hypersurface in $\PP(S^eV^*)$ if $e\ge 2$. The degree of $\sZ_{e,\sX}$ is
\begin{equation}\label{eq: degree dual Veronese X}
\sum_{j=0}^n(-1)^j(n+1-j)\,c_j^M(\sX)\cdot(eh)^{n-j}\,,
\end{equation}
where $h=c_1(\sO_\sX(1))$.
\end{proposition}
\begin{proof}
Every element $\sY\in\PP(S^eV^*)$ corresponds to a hyperplane $H_\sY\subset\PP(S^eV)$. In particular, if we consider the Veronese variety $\sV_{e,N}=\nu_e(\PP^N)\subset\PP(S^eV)$, then $\sY\cong\sV_{e,N}\cap H_\sY$. Now assume that $T_x\sX\subset T_x\sY$ for some point $x\in \sX_{\mathrm{sm}}\cap \sY_{\mathrm{sm}}$. This is equivalent to the condition that the hyperplane $H_\sY\subset\PP(S^eV)$ contains the tangent space $T_{\nu_e(x)}\sX_e$, namely $H$ belongs to the dual variety $\sX_e^\vee$ of $\sX_e$. By taking closures, the variety $\sZ_{e,\sX}$ is equal to $\sX_e^\vee$. Since we are assuming that $\sX$ is irreducible, $\sX_e$ is irreducible as well, and its dual variety $\sX_e^\vee$ too.

Now observe that any subvariety of $\sV_{e,N}$ is of the form $\sX_e$ for some variety $\sX\subset\PP^N$, and $\deg(\sX_e)=e^n\deg(\sX)$, in particular, it is larger than one when $e>1$. This means that $\sV_{e,N}$ does not contain linear subspaces for all $e\ge 2$, hence its dual variety $\sV_{e,N}^\vee$ is always a hypersurface. This fact is also an immediate consequence of Lemma \ref{lem: conditions for X mu dual hypersurface, in general}.
Similarly, for any subvariety $\sX\subset\PP^N$ and $e\ge 2$, the variety $\sX_e^\vee=\sZ_{e,\sX}$ is a hypersurface.

Now assume that $\sX$ is smooth. The first Chern class of the line bundle corresponding to the Veronese embedding of $\sX$ into $\PP(S^eV)$ is equal to $eh$. Using this fact and applying equation \eqref{eq: degdual}, we obtain the degree of $\sX_e^\vee$ in \eqref{eq: degree dual Veronese X} for all $e\ge 2$. A similar formula holds when $\sX$ is singular, after replacing the Chern classes with Chern-Mather classes.
\end{proof}

\begin{example}\label{ex: quartic surface}
Consider homogeneous coordinates $(z_{30},z_{21},z_{12},z_{03})$ for $\PP^3$ and the quartic surface $\sX\subset\PP^3$ defined by the equation
\begin{equation}
z_{21}^{2}z_{12}^{2}-4\,z_{30}z_{12}^{3}-4\,z_{21}^{3}z_{03}+18\,z_{30}z_{21}z_{12}z_{03}-27\,z_{30}^{2}z_{03}^{2}=0\,,
\end{equation}
which defines the discriminant variety of a cubic binary form.
The surface $\sX$ has dual defect $\mathrm{def}(\sX)=1$, as its dual variety is a rational normal cubic in $\PP^3$. Consider the smooth point $P=[0,1,0,0]$ of $\sX$. The tangent plane $T_P\sX$ has equation $z_{03}=0$, and its intersection with $\sX$ is scheme-theoretically the union of the double line $\V(z_{03},z_{12}^2)$ and the simple conic $\V(z_{03},z_{21}^2-z_{30}z_{12})$.
Furthermore, the conic $\V(z_{03},z_{21}^2-z_{30}z_{12})$ is tangent to the line $\V(z_{03},z_{12})$ at the point $[1,0,0,0]$, which belongs to $\sX_{\mathrm{sing}}$.
Therefore, $\mathrm{Cont}(T_P\sX,\sX)$ is the line $\V(z_{03},z_{12})$, confirming the fact that $\mathrm{def}(\sX)=1$. Now consider a quadric surface $\sQ\subset\PP^3$. Its general equation has the form
\begin{equation}
q_{30,30}z_{30}^2+q_{30,21}z_{30}z_{21}+\cdots+q_{03,03}z_{03}^2=0\,,
\end{equation}
where the coefficients $q_{30,30},q_{30,21},\ldots q_{03,03}$ are homogeneous coordinates of $\PP(S^2(\C^4)^*)\cong\PP^9$. The normal space of $\sQ$ at $P$ is generated by the vector $(q_{30,21},2q_{21,21},q_{12,21},q_{03,21})$. Therefore, $T_P\sQ=T_P\sX$ if and only if the previous vector is proportional to the vector $(0,0,0,1)$, namely if and only if $q_{30,21}=q_{21,21}=q_{12,21}=0$. These conditions define a projective subspace $W_P$ of $\PP(S^2(\C^4)^*)$ of dimension 6.
A generic element $\sQ\in W_P$ meets $\sX$ on an irreducible curve of degree $8$, and the reduced singular locus of $\sQ\cap\sX$ consists of $P$ and the six points of intersection between $\sQ$ and the rational normal cubic $\sX_{\mathrm{sing}}$.
All these considerations tell us that a generic element of $W_P$ is tangent to $\sX$ only at $P$. This confirms that the variety $\sZ_{2,\sX}=[\nu_2(\sX)]^\vee$ is a hypersurface in $\PP(S^2(\C^4)^*)$.
In order to compute its degree, note that $\delta_0(\sX)=0$, $\delta_1(\sX)=\deg(\sX^{\vee})=3$, and $\delta_2(\sX)=\deg(\sX)=4$. 
By \cite[Thm. $3$]{piene1988cycles} one expresses the Chern-Mather classes of $\sX$ via its polar classes (cf. \eqref{eq: Holme inverted}) as
\begin{align*}\label{eq: polar Chern-Mather}
c_i^M(\sX) &= \sum_{j=0}^i(-1)^{j}\binom{n+1-j}{i-j}\delta_{n-j}(\sX)\cdot h^{i-j}\,.
\end{align*}
Using this, we obtain that $c_0^M(\sX)=4$, $c_1^M(\sX)=9$, and $c_2^M(\sX)=6$. Then using \eqref{eq: degree dual Veronese X}, we get that
\begin{equation}
\deg(\sZ_{2,\sX})=3\,c_0^M(\sX)\deg((2h)^2)-2\,c_1^M(\sX)\deg(2h)+c_2^M(\sX)=3\cdot 4\cdot 4-2\cdot 9\cdot 2+6=18\,.
\end{equation}\hfill$\diamondsuit$
\end{example}

All the previous considerations lead to the following result.

\begin{proposition}\label{pro: nonzero leadcoeff ED polynomial Q symbolic}
Let $\sX$ be an irreducible projective variety in $\PP(V)$.
Consider the nonzero polynomial $p_{\gEDeg(\sX)}(Q)$ in \eqref{eq: write extended ED polynomial explicitly} and the vanishing locus $\V(p_{\gEDeg(\sX)})\subset\PP(S^2V)$. Then,
\begin{equation}\label{eq: zero locus leadcoef 1}
 \V(p_{\gEDeg(\sX)})=\PP(\{Q\in S^2V\mid\EDeg_Q(\sX)<\gEDeg(\sX)\})\,.
\end{equation}
In particular, $p_{\gEDeg(\sX)}(Q)$ has a positive degree in the entries of $M_Q$.
Furthermore, $\V(p_{\gEDeg(\sX)})$ is an irreducible hypersurface and corresponds to the dual variety $[\nu_2(\sX)]^\vee$ under the second Veronese embedding of $\PP(V)$ into $\PP(S^2V)$.
\end{proposition}
\begin{proof}
Identity \eqref{eq: zero locus leadcoef 1} follows immediately by Proposition \ref{prop: degree and leadcoef extended ED polynomial}. Since the right-hand side is a proper closed subvariety of $\PP(S^2V)$, then necessarily $p_{\gEDeg(\sX)}(Q)$ has positive degree in the entries of $M_Q$.
Consider the variety $\sZ_{e,\sX}$ introduced in \eqref{eq: def Z d X} for $e=2$.
The inclusion $\V(p_{\gEDeg(\sX)})\subset \sZ_{2,\sX}$ follows by identity \eqref{eq: zero locus leadcoef 1} and Theorem \ref{thm: ED degree sum polar classes}.
Since $\V(p_{\gEDeg(\sX)})$ is a hypersurface, and $\sZ_{2,\sX}$ is an irreducible hypersurface by Proposition \ref{prop: variety Q not transversal X irreducible}, then necessarily $\V(p_{\gEDeg(\sX)})=\sZ_{2,\sX}$.
The last part results from Proposition \ref{prop: variety Q not transversal X irreducible}.
\end{proof}

For the rest of the section, we focus on the case when $\sX$ is a Segre-Veronese variety $\sV_{\bd,\bn}$.
Recall the notations given at the beginning of Section \ref{sec: product metrics tensor spaces}.

\begin{lemma}{\cite[Chapter 1, Corollary 5.10]{GKZ}}\label{lem: conditions for X mu dual hypersurface, in general}
The dual variety of the Segre-Veronese variety $\sV_{\bd,\bn}$ is a hypersurface in $\PP(S^\bd V^*)$ if and only if
\begin{equation}\label{eq: inequalities for X mu}
n_i\le\sum_{j\neq i}n_j
\end{equation}
for all indices $i$ such that $d_i=1$.
\end{lemma}

The next lemma is a consequence of the results in \cite[\S1.4.B]{tevelev2003projectively} and is useful when considering the composition of a Segre-Veronese embedding with another Veronese embedding.

\begin{lemma}\label{lem: composing with a Veronese embedding}
Consider the Segre-Veronese variety $\sV_{\bd,\bn}\subset\PP(S^\bd V)$. For all $e\ge 2$, the Veronese embedding $\nu_e(\sV_{\bd,\bn})\subset\PP(S^e(S^\bd V))$ is degenerate, in particular, it is contained in a proper subspace $W$ of $\PP(S^e(S^\bd V))$ isomorphic to $\PP(S^{e\bd} V)$, where $e\bd=(ed_1,\ldots,ed_k)$.
Furthermore, the dual variety $\nu_e(\sV_{\bd,\bn})^\vee$ in $\PP(S^e(S^\bd V))$ can be described as the cone over the dual variety $\sV_{e\bd,\bn}^\vee$ in $\PP(S^{e\bd}V)$ with vertex corresponding to $W$.
\end{lemma}

In the proof of Theorem \ref{thm: gen ED degree Segre-Veronese} we computed the degrees of the Chern classes $c_i(\sV_{\bd,\bn})$. These allow us to derive also the polar degrees of $\sV_{\bd,\bn}$:

\begin{proposition}\label{pro: polar degrees Segre-Veronese}
The polar degrees of the Segre-Veronese variety $\sV_{\bd,\bn}$ are
\begin{equation}\label{eq: polar classes Segre-Veronese}
\delta_j(\sV_{\bd,\bn})=\sum_{i=0}^{|\bn|-j}(-1)^{i}\binom{|\bn|+1-i}{j+1}(|\bn|-i)!\sum_{|\balpha|=i}\gamma_{\balpha}\,\bd^{\bn-\balpha}\quad\forall\,j\in\{0,\ldots,|\bn|\}\,,
\end{equation}
where $\gamma_{\balpha}$ is defined in \eqref{eq: def gamma alpha}.
When the dual variety $\sV_{\bd,\bn}^\vee$ is a hypersurface, its degree is
\begin{equation}\label{eq: degree dual Segre-Veronese}
\deg(\sV_{\bd,\bn}^\vee)=\delta_0(\sV_{\bd,\bn})=\sum_{i=0}^{|\bn|}(-1)^i(|\bn|+1-i)!\sum_{|\balpha|=i}\gamma_{\balpha}\,\bd^{\bn-\balpha}\,.
\end{equation}
\end{proposition}
\begin{proof}
Since $\sV_{\bd,\bn}$ is a smooth variety, the degrees $\delta_j(\sV_{\bd,\bn})$ may be written in terms of the Chern classes of $\sV_{\bd,\bn}$ computed in Proposition \ref{prop: degrees Chern classes Segre-Veronese} and using the identity \eqref{eq: Holme}. In particular, when $\sV_{\bd,\bn}^\vee$ is a hypersurface, its degree coincides with the polar class $\delta_0(\sV_{\bd,\bn})$, so we apply the formula \eqref{eq: degdual}.
\end{proof}

\begin{corollary}\label{cor: dual variety e-Veronese embedding V d,n}
Consider the Segre-Veronese variety $\sV_{\bd,\bn}\subset\PP(S^\bd V)$. The locus of hypersurfaces of degree $e$ that have a nontransversal intersection with $\sV_{\bd,\bn}$ is (isomorphic to) the dual variety $[\nu_e(\sV_{\bd,\bn})]^\vee$ in $\PP(S^e(S^\bd V))$. Furthermore, it is always a hypersurface when $e\ge 2$, of degree
\begin{equation}\label{eq: degree dual Segre-Veronese d mu}
 \deg([\nu_e(\sV_{\bd,\bn})]^\vee)=\sum_{i=0}^{|\bn|}(-1)^i e^{|\bn|-i}(|\bn|+1-i)!\sum_{|\balpha|=i}\gamma_{\balpha}\bd^{\bn-\balpha}\,,
\end{equation}
where $\gamma_{\balpha}$ is defined in \eqref{eq: def gamma alpha}.
\end{corollary}
\begin{proof}
The first part follows by Proposition \ref{prop: variety Q not transversal X irreducible}.
Consider now the nondegenerate Segre-Veronese variety $\sV_{e\bd,\bn}\subset\PP(S^{e\bd}V)$. By Lemma \ref{lem: conditions for X mu dual hypersurface, in general}, the dual variety of $\sV_{e\bd,\bn}$ is a hypersurface for all $e\ge 2$. This means that, applying Lemma \ref{lem: composing with a Veronese embedding}, the dual variety $[\nu_e(\sV_{\bd,\bn})]^\vee$ in $\PP(S^e(S^\bd V))$ is a hypersurface as well, with the same degree of $\sV_{e\bd,\bn}^\vee$, computed using Proposition \ref{pro: polar degrees Segre-Veronese}.
\end{proof}

We consider two applications of Corollary \ref{cor: dual variety e-Veronese embedding V d,n}.

\begin{example}
Let $k=1$, $\bd=(d)$ and $\bn=(1)$, hence we are considering the rational normal curve $\sV_{d,1}\subset\PP(S^d\C^2)=\PP^d_{\mC}$. Then
\begin{equation}\label{eq: degree dual nu_e(V_{d,1})}
\deg([\nu_e(\sV_{d,1})]^\vee) = \deg(\sV_{ed,1}^\vee) = \sum_{i=0}^1(-1)^i(ed)^{1-i}(2-i)!\gamma_i = 2ed\gamma_0-\gamma_1 = 2(ed-1)\,,
\end{equation}
since in this case $\gamma_0=1$ and $\gamma_1=2$. Note that $2(ed-1)$ is the degree of the discriminant of a binary form of degree $ed$.
In particular, for $e=2$, we obtain that the leading coefficient of the extended ED polynomial $\mathrm{EDpoly}_{\sV_{d,1}}(u,\varepsilon^2,Q)$ has degree $4d-2$.\hfill$\diamondsuit$
\end{example}

\begin{example}
Let $k=2$, $\bd=(d_1,d_2)$ and $\bn=(1,1)$. Then
\begin{align}
\begin{split}
 \deg([\nu_e(\sV_{\bd,\bn})]^\vee) &= \deg(\sV_{e\bd,\bn}^\vee)\\
 &= \sum_{i=0}^2(-1)^ie^{2-i}(3-i)!\sum_{|\balpha|=i}\gamma_{\balpha}\bd^{\bn-\balpha}\\
 &= 6e^2\gamma_{(0,0)}d_1d_2-2e(\gamma_{(1,0)}d_2+\gamma_{(0,1)}d_1)+\gamma_{(1,1)}\\
 &= 6e^2d_1d_2-4e(d_1+d_2)+4\,,
\end{split}
\end{align}
since in this case $\gamma_{(0,0)}=1$, $\gamma_{(1,0)}=\gamma_{(0,1)}=2$, and $\gamma_{(1,1)}=4$. In particular, for $\bd=(1,1)$ we are considering the Segre variety $\Sigma_\bn=\Sigma_2$ and
\[
\deg([\nu_e(\Sigma_2)]^\vee) = 6e^2-8e+4\,.
\]
As a consequence, the leading coefficient of the extended ED polynomial $\mathrm{EDpoly}_{\Sigma_2}(u,\varepsilon^2,Q)$ is a homogeneous polynomial in 10 variables of degree 12.\hfill$\diamondsuit$
\end{example}

Our main Conjecture \ref{conj: main} has a natural counterpart in terms of extended ED polynomials of Segre-Veronese varieties.

\begin{corollary}\label{cor: main conj revisited}
Consider the Segre-Veronese variety $\sV_{\bd,\bn}\subset\PP(S^\bd V)$ and the expansion of its extended ED polynomial
\[
\mathrm{EDpoly}_{\sV_{\bd,\bn}}(u,Q,\varepsilon^2) = \sum_{i=0}^{\gEDeg(\sV_{\bd,\bn})}p_i(u,Q)\,\varepsilon^{2i}\,.
\]
For all $0\le i\le \gEDeg(\sV_{\bd,\bn}))$, define the affine cone in $S^2(S^{\bd} V)$
\begin{equation}\label{eq: def varieties vanishing coefficients ED poly}
\sW_{\sX}^{(i)} \coloneqq \bigcap_{j=0}^i\{Q\in S^2(S^{\bd} V) \mid \text{$p_{\gEDeg(\sV_{\bd,\bn})-j}(u,Q)=0$ for all $u\in S^{\bd} V$}\}\,.
\end{equation}
Then Conjecture \ref{conj: main} is true if and only if
\begin{equation}\label{eq: empty intersection cone pd matrices}
\sW_{\sX}^{(i)} \cap S^2_+(S^{\bd} V^{\mR}) = \emptyset\quad\forall\,i\ge\EDef_{Q_F}(\sV_{\bd,\bn})\,,
\end{equation}
where $Q_F\in S^2_+(S^{\bd} V^{\mR})$ is a Frobenius inner product.
\end{corollary}

\bibliographystyle{alpha}
\bibliography{bibliography.bib}

\newcommand{\etalchar}[1]{$^{#1}$}
\begin{thebibliography}{SDLF{\etalchar{+}}17}

\bibitem[AH18]{aluffi-harris}
P.~Aluffi and C.~Harris.
\newblock The {E}uclidean distance degree of smooth complex projective
  varieties.
\newblock {\em Algebra Number Theory}, 12(8):2005--2032, 2018.

\bibitem[Alu18]{aluffi2018projective}
P.~Aluffi.
\newblock Projective duality and a {C}hern-{M}ather involution.
\newblock {\em Trans. Amer. Math. Soc.}, 370(3):1803--1822, 2018.

\bibitem[Ban38]{Banach}
S.~Banach.
\newblock {\"Uber homogene Polynome in $(L^2)$}.
\newblock {\em Studia Math.}, 7:36--44, 1938.

\bibitem[BH04]{BH04}
A.~Banyaga and D.~Hurtubise.
\newblock {\em Lectures on Morse Homology}.
\newblock Texts in the Mathematical Sciences. Springer Dordrecht, 2004.

\bibitem[Bre19]{Breiding2017HowME}
P.~Breiding.
\newblock How many eigenvalues of a random symmetric tensor are real?
\newblock {\em Trans. Amer. Math. Soc.}, 372(11):7857--7887, 2019.

\bibitem[Com14]{Co}
P.~Comon.
\newblock Tensors: A brief introduction.
\newblock {\em IEEE Signal Processing Magazine}, 31(3):44--53, 2014.

\bibitem[dCLdA15]{daSilva2015iterative}
A.~{da Silva}, P.~Comon, and A.~L.F.~de Almeida.
\newblock An iterative deflation algorithm for exact {CP} tensor decomposition.
\newblock In {\em 2015 IEEE International Conference on Acoustics, Speech and
  Signal Processing (ICASSP)}, pages 3961--3965, 2015.

\bibitem[DFLW17]{Derksen2017TheoreticalAC}
H.~Derksen, S.~Friedland, L.-H. Lim, and L.~Wang.
\newblock Theoretical and computational aspects of entanglement.
\newblock \arxiv{1705.07160}, 2017.

\bibitem[DH16]{DH2016}
J.~Draisma and E.~Horobe\c{t}.
\newblock The average number of critical rank-one approximations to a tensor.
\newblock {\em Linear Multilinear Algebra}, 64(12):2498--2518, 2016.

\bibitem[DHO{\etalchar{+}}16]{DHOST}
J.~Draisma, E.~Horobe\c{t}, G.~Ottaviani, B.~Sturmfels, and R.~Thomas.
\newblock The {E}uclidean distance degree of an algebraic variety.
\newblock {\em Found. Comput. Math.}, 16(1):99--149, 2016.

\bibitem[dKL22]{KlerkLaurent}
E.~de~Klerk and M.~Laurent.
\newblock Convergence analysis of a {L}asserre hierarchy of upper bounds for
  polynomial minimization on the sphere.
\newblock {\em Math. Program.}, 193:665--685, 2022.

\bibitem[DM18]{DM2018}
H.~Derksen and V.~Makam.
\newblock Highly entangled tensors.
\newblock {\em Linear and Multilinear Algebra}, 2018.

\bibitem[DOT18]{DOT}
J.~Draisma, G.~Ottaviani, and A.~Tocino.
\newblock Best rank-{$k$} approximations for tensors: generalizing
  {E}ckart-{Y}oung.
\newblock {\em Res. Math. Sci.}, 5(2):27, 2018.

\bibitem[dSL08]{deSilva2008tensor}
V.~de~Silva and L.~Lim.
\newblock Tensor rank and the ill-posedness of the best low-rank approximation
  problem.
\newblock {\em SIAM J. Matrix Anal. Appl.}, 30(3):1084--1127, 2008.

\bibitem[DV04]{DELATHAUWER200431}
L.~{De Lathauwer} and J.~Vandewalle.
\newblock Dimensionality reduction in higher-order signal processing and
  rank-($r_1,r_2,\dots,r_n$) reduction in multilinear algebra.
\newblock {\em Linear Algebra and its Applications}, 391:31--55, 2004.
\newblock Special Issue on Linear Algebra in Signal and Image Processing.

\bibitem[EY36]{EckartYoung}
C.~Eckart and G.~Young.
\newblock The approximation of one matrix by another of lower rank.
\newblock {\em Psychometrika}, 1:211--218, 1936.

\bibitem[FMPS11]{Friedland2011OnBR}
S.~Friedland, V.~Mehrmann, R.~Pajarola, and S.~K. Suter.
\newblock On best rank one approximation of tensors.
\newblock {\em Numerical Linear Algebra with Applications}, 20, 2011.

\bibitem[FMS21]{fevola2021pencils}
C.~Fevola, Y.~Mandelshtam, and B.~Sturmfels.
\newblock Pencils of quadrics: old and new.
\newblock {\em Matematiche (Catania)}, 76(2):319--335, 2021.

\bibitem[FO14]{FO}
S.~Friedland and G.~Ottaviani.
\newblock The number of singular vector tuples and uniqueness of best rank-one
  approximation of tensors.
\newblock {\em Found. Comput. Math.}, 14(6):1209--1242, 2014.

\bibitem[Fri11]{Friedland2011BestRO}
S.~Friedland.
\newblock Best rank one approximation of real symmetric tensors can be chosen
  symmetric.
\newblock {\em Frontiers of Mathematics in China}, 8:19--40, 2011.

\bibitem[Ful98]{fulton1998intersection}
W.~Fulton.
\newblock {\em Intersection theory}, volume~2 of {\em Ergebnisse der Mathematik
  und ihrer Grenzgebiete. 3. Folge. A Series of Modern Surveys in Mathematics
  [Results in Mathematics and Related Areas. 3rd Series. A Series of Modern
  Surveys in Mathematics]}.
\newblock Springer-Verlag, Berlin, second edition, 1998.

\bibitem[GG73]{GG:stablemaps}
M.~Golubitsky and V.~Guillemin.
\newblock {\em Stable mappings and their singularities}.
\newblock Graduate Texts in Mathematics, Vol. 14. Springer-Verlag, New
  York-Heidelberg, 1973.

\bibitem[GJZ17]{GJZh}
R.~Ge, C.~Jin, and Y.~Zheng.
\newblock No spurious local minima in nonconvex low rank problems: A unified
  geometric analysis.
\newblock In Doina Precup and Yee~Whye Teh, editors, {\em Proceedings of the
  34th International Conference on Machine Learning}, volume~70 of {\em
  Proceedings of Machine Learning Research}, pages 1233--1242. PMLR, 06--11 Aug
  2017.

\bibitem[GKT13]{GKT2013}
L.~Grasedyck, D.~Kressner, and C.~Tobler.
\newblock A literature survey of low-rank tensor approximation techniques.
\newblock {\em GAMM-Mitteilungen}, 36(1):53--78, 2013.

\bibitem[GKZ94]{GKZ}
I.~Gel'fand, M.~Kapranov, and A.~Zelevinsky.
\newblock {\em Discriminants, resultants, and multidimensional determinants}.
\newblock Mathematics: Theory \& Applications. Birkh\"{a}user Boston, Inc.,
  Boston, MA, 1994.

\bibitem[GM88]{GorMac-Morse}
M.~Goresky and R.~MacPherson.
\newblock {\em Stratified {Morse} theory}, volume~14 of {\em Ergeb. Math.
  Grenzgeb., 3. Folge}.
\newblock Berlin etc.: Springer-Verlag, 1988.

\bibitem[GM22]{ge2022optimization}
R.~Ge and T.~Ma.
\newblock On the optimization landscape of tensor decompositions.
\newblock {\em Mathematical Programming}, 193:713--759, 2022.

\bibitem[GS97]{GS}
D.~Grayson and M.~Stillman.
\newblock Macaulay 2--a system for computation in algebraic geometry and
  commutative algebra, 1997.

\bibitem[HJ13]{HJ}
R.~Horn and C.~Johnson.
\newblock {\em Matrix analysis}.
\newblock Cambridge University Press, Cambridge, second edition, 2013.

\bibitem[HL13]{HiLi}
C.~J. Hillar and L.-H. Lim.
\newblock Most tensor problems are {NP}-hard.
\newblock {\em J. ACM}, 60(6), 2013.

\bibitem[Hol88]{Holme}
A.~Holme.
\newblock The geometric and numerical properties of duality in projective
  algebraic geometry.
\newblock {\em Manuscripta Math.}, 61(2):145--162, 1988.

\bibitem[HT25]{HTT2023}
E.~Horobe\c{t} and E.~T. Turatti.
\newblock When does subtracting a rank-one approximation decrease tensor rank?
\newblock {\em Linear Algebra Appl.}, 709:397--415, 2025.

\bibitem[HW19]{horobet2019offset}
E.~Horobe\c{t} and M.~Weinstein.
\newblock Offset hypersurfaces and persistent homology of algebraic varieties.
\newblock {\em Comput. Aided Geom. Design}, 74:101767, 14, 2019.

\bibitem[KB09]{KoldaBader}
T.~G. Kolda and B.~W. Bader.
\newblock Tensor decompositions and applications.
\newblock {\em SIAM Review}, 51(3):455--500, 2009.

\bibitem[Kle86]{kleiman1986tangency}
S.~Kleiman.
\newblock Tangency and duality.
\newblock In {\em Proceedings of the 1984 {V}ancouver conference in algebraic
  geometry}, volume~6 of {\em CMS Conf. Proc.}, pages 163--225. Amer. Math.
  Soc., Providence, RI, 1986.

\bibitem[KMMT22]{kohn2022geometry}
K.~Kohn, T.~Merkh, G.~Mont\'{u}far, and M.~Trager.
\newblock Geometry of linear convolutional networks.
\newblock {\em SIAM J. Appl. Algebra Geom.}, 6(3):368--406, 2022.

\bibitem[KMST24]{kohn2023function}
K.~Kohn, G.~Mont\'ufar, V.~Shahverdi, and M.~Trager.
\newblock Function space and critical points of linear convolutional networks.
\newblock {\em SIAM J. Appl. Algebra Geom.}, 8(2):333--362, 2024.

\bibitem[Kos00]{Kostlan}
E.~Kostlan.
\newblock On the expected number of real roots of a system of random polynomial
  equations.
\newblock {\em Foundations of Computational Mathematics}, pages 149--188, 2000.

\bibitem[Koz17]{Kozhasov2017OnFR}
K.~Kozhasov.
\newblock On fully real eigenconfigurations of tensors.
\newblock {\em SIAM J. Appl. Algebra Geom.}, 2:339--347, 2017.

\bibitem[KR01]{KR}
E.~Kofidis and P.~A. Regalia.
\newblock Tensor approximation and signal processing applications.
\newblock {\em Structured Matrices in Mathematics, Computer Science and
  Engineering}, I, 2001.

\bibitem[KSS25]{kohn2023geometry}
K.~Kohn, A.-L. Sattelberger, and V.~Shahverdi.
\newblock Geometry of linear neural networks: equivariance and invariance under
  permutation groups.
\newblock {\em SIAM J. Matrix Anal. Appl.}, 46(2):1378--1415, 2025.

\bibitem[K{\"u}n23]{Kuenneth}
H.~K{\"u}nneth.
\newblock {\"U}ber die {B}ettischen {Z}ahlen einer {P}roduktmannigfaltigkeit.
\newblock {\em Mathematische Annalen}, 90:65--85, 1923.

\bibitem[MH99]{melle-euler}
A.~Melle-Hern\'{a}ndez.
\newblock Euler characteristic of the {M}ilnor fibre of plane singularities.
\newblock {\em Proc. Amer. Math. Soc.}, 127(9):2653--2655, 1999.

\bibitem[Mil63]{Milnor1963}
J.~Milnor.
\newblock {\em Morse Theory}.
\newblock Princeton University Press, Princeton, 1963.

\bibitem[Mil68]{milnor-book-sings}
J.~Milnor.
\newblock {\em Singular points of complex hypersurfaces}, volume~61 of {\em
  Ann. Math. Stud.}
\newblock Princeton University Press, Princeton, NJ, 1968.

\bibitem[MRW20]{maxim2020defect}
L.~Maxim, J.~Rodriguez, and B.~Wang.
\newblock Defect of {E}uclidean distance degree.
\newblock {\em Advances in Applied Mathematics}, 121:102101, oct 2020.

\bibitem[MS74]{milnor1974characteristic}
J.~W. Milnor and J.~D. Stasheff.
\newblock {\em Characteristic Classes}.
\newblock Annals of mathematics studies. Princeton University Press, 1974.

\bibitem[Nes03]{Nesterov}
Yu. Nesterov.
\newblock Random walk in a simplex and quadratic optimization over convex
  polytopes.
\newblock LIDAM Discussion Papers CORE 2003071, Univ. Cath. Louvain, 2003.

\bibitem[NW13]{Nie2013SemidefiniteRF}
J.~Nie and L.~Wang.
\newblock Semidefinite relaxations for best rank-1 tensor approximations.
\newblock {\em SIAM J. Matrix Anal. Appl.}, 35:1155--1179, 2013.

\bibitem[OP15]{ottaviani2015geometric}
G.~Ottaviani and R.~Paoletti.
\newblock A geometric perspective on the singular value decomposition.
\newblock {\em Rend. Istit. Mat. Univ. Trieste}, 47:107--125, 2015.

\bibitem[OS20]{ottaviani2020distance}
G.~Ottaviani and L.~Sodomaco.
\newblock The distance function from a real algebraic variety.
\newblock {\em Computer Aided Geometric Design}, 82:101927, oct 2020.

\bibitem[OSV21]{ottaviani2021asymptotics}
G.~Ottaviani, L.~Sodomaco, and E.~Ventura.
\newblock Asymptotics of degrees and {ED} degrees of {S}egre products.
\newblock {\em Adv. in Appl. Math.}, 130:Paper No. 102242, 36, 2021.

\bibitem[Pie78]{piene1978polar}
R.~Piene.
\newblock Polar classes of singular varieties.
\newblock {\em Ann. Sci. \'{E}cole Norm. Sup. (4)}, 11(2):247--276, 1978.

\bibitem[Pie88]{piene1988cycles}
R.~Piene.
\newblock Cycles polaires et classes de {C}hern pour les vari\'{e}t\'{e}s
  projectives singuli\`eres.
\newblock In {\em Introduction \`a la th\'{e}orie des singularit\'{e}s, {II}},
  volume~37 of {\em Travaux en Cours}, pages 7--34. Hermann, Paris, 1988.

\bibitem[QCL16]{QCL16}
Y.~Qi, P.~Comon, and L.~Lim.
\newblock Uniqueness of nonnegative tensor approximations.
\newblock {\em IEEE Transactions on Information Theory}, 62(4):2170--2183,
  2016.

\bibitem[QL17]{QiLuo}
L.~Qi and Z.~Luo.
\newblock {\em Tensor Analysis: Spectral Theory and Special Tensors}.
\newblock Society for Industrial and Applied Mathematics, Philadelphia, PA,
  2017.

\bibitem[RW17]{rod-wang-maxlike}
J.~Rodriguez and B.~Wang.
\newblock The maximum likelihood degree of mixtures of independence models.
\newblock {\em SIAM J. Appl. Algebra Geom.}, 1(1):484--506, 2017.

\bibitem[Sch02]{Schatzman}
M.~Schatzman.
\newblock {\em Numerical Analysis : A Mathematical Introduction}.
\newblock Oxford: Clarendon Press, 2002.

\bibitem[SDLF{\etalchar{+}}17]{SDFHPF}
N.~Sidiropoulos, L.~De~Lathauwer, X.~Fu, K.~Huang, E.~Papalexakis, and
  C.~Faloutsos.
\newblock Tensor decomposition for signal processing and machine learning.
\newblock {\em IEEE Transactions on Signal Processing}, 65(13):3551--3582,
  2017.

\bibitem[Shi16]{shin-bound}
J.~Shin.
\newblock A bound for the {M}ilnor sum of projective plane curves in terms of
  {GIT}.
\newblock {\em J. Korean Math. Soc.}, 53(2):461--473, 2016.

\bibitem[Slo]{oeis}
N.~Sloane.
\newblock The on-line encyclopedia of integer sequences.
\newblock Available at \url{http://oeis.org}.

\bibitem[Sod20]{sodphd}
L.~Sodomaco.
\newblock {\em The Distance Function from the Variety of partially symmetric
  rank-one Tensors}.
\newblock PhD thesis, Universit\`a degli Studi di Firenze, 2020.
\newblock Available at \url{https://flore.unifi.it/handle/2158/1220535?}

\bibitem[{Sta}23]{stacks-project}
The {Stacks project authors}.
\newblock The stacks project.
\newblock \url{https://stacks.math.columbia.edu}, 2023.

\bibitem[Stu21]{Stu22icm}
B.~Sturmfels.
\newblock Beyond linear algebra.
\newblock {\em \arxiv{2108.09494}}, 2021.
\newblock Presented at the International Congress of Mathematicians 2022.

\bibitem[Tev03]{tevelev2003projectively}
E.~Tevelev.
\newblock Projectively dual varieties.
\newblock {\em Journal of Mathematical Sciences}, 117(6):4585--4732, 2003.

\bibitem[Tho91]{Thomp}
R.~Thompson.
\newblock Pencils of complex and real symmetric and skew matrices.
\newblock {\em Linear Algebra Appl.}, 147:323--371, 1991.

\bibitem[UT17]{Udell2017WhyAB}
M.~Udell and A.~Townsend.
\newblock Why are big data matrices approximately low rank?
\newblock {\em SIAM J. Math. Data Sci.}, 1:144--160, 2017.

\bibitem[vS89]{vanstraten1989note}
D.~van {S}traten.
\newblock A note on the number of periodic orbits near a resonant equilibrium
  point.
\newblock {\em Nonlinearity}, 2(3):445--458, 1989.

\bibitem[WG03]{Wei2003GeometricMO}
T.~Wei and P.~Goldbart.
\newblock Geometric measure of entanglement and applications to bipartite and
  multipartite quantum states.
\newblock {\em Physical Review A}, 68:042307, 2003.

\bibitem[Zha18]{zhang2018chern}
X.~Zhang.
\newblock Chern classes and characteristic cycles of determinantal varieties.
\newblock {\em Journal of Algebra}, 497:55--91, 2018.

\end{thebibliography}

\end{document}